\documentclass[11pt]{article}
\usepackage{amsfonts,amscd,amssymb, amsmath}
\usepackage[usenames]{color}
\usepackage{graphicx}
\usepackage{euscript}
\usepackage[bookmarksnumbered,plainpages,driverfallback=dvipdfm,backref=page]{hyperref}
\usepackage{enumerate}
\usepackage[all,tips]{xy}
\usepackage{epic,eepic}
\usepackage{color}
\usepackage{tikz}

\usepackage[hyperpageref]{backref}
\backref
\renewcommand*{\backref}[1]{}  
\renewcommand*{\backrefalt}[4]{
  \ifcase #1 %
  \relax
  \or
(Cited page~#2.)%
  \else
(Cited pages~#2.)%
  \fi}

\newtheorem{definition}{Definition}[section]
\newtheorem{lemma}[definition]{Lemma}
\newtheorem{proposition}[definition]{Proposition}
\newtheorem{corollary}[definition]{Corollary}

\newtheorem{theorem}[definition]{Theorem}
\newtheorem{example}[definition]{Example}

\def\a{\alpha}
\def\b{\beta}

\def\rawo\lonra{\longrightarrow}

\def\ot{\otimes}

\allowdisplaybreaks[4]
\newenvironment{proof}{{\it Proof.}}{\hfill $ \square $ \vskip 4mm}

\begin{document}

\title{Tensor products and perturbations of \\
BiHom-Novikov-Poisson algebras
}
\author{Ling Liu \\
College of Mathematics and Computer Science,\\
Zhejiang Normal University, \\
Jinhua 321004, China \\
e-mail: ntliulin@zjnu.cn \and Abdenacer Makhlouf \\
Universit\'{e} de Haute Alsace, \\
IRIMAS-D\'epartement  de Math\'{e}matiques,  \\
6 bis, rue des fr\`{e}res Lumi\`{e}re, F-68093 Mulhouse, France\\
e-mail: Abdenacer.Makhlouf@uha.fr \and Claudia Menini \\
University of Ferrara, \\
Department of Mathematics and Computer Science\\
Via Machiavelli 30, Ferrara, I-44121, Italy\\
e-mail: men@unife.it \and Florin Panaite \\
Institute of Mathematics of the Romanian Academy,\\
PO-Box 1-764, RO-014700 Bucharest, Romania\\
e-mail: florin.panaite@imar.ro }
\date{}
\maketitle

\begin{abstract}
We study BiHom-Novikov-Poisson algebras, which are twisted generalizations of Novikov-Poisson algebras and
Hom-Novikov-Poisson algebras,
and find that BiHom-Novikov-Poisson algebras are closed under tensor products and several kinds of perturbations.
Necessary and sufficient conditions are given under which BiHom-Novikov-Poisson algebras give rise to
BiHom-Poisson algebras. \\

\begin{small}
\end{small}
\end{abstract}

\section*{Introduction}

Hom-type algebras appeared in the Physics literature of the 1990's, when looking for quantum deformations
of some algebras of vector fields, like Witt and Virasoro algebras
(\cite{AizawaSaito}, \cite{Hu}). It was observed that algebras obtained by deforming certain Lie algebras
no longer satisfied the Jacobi identity, but a
modified version of it involving a homomorphism. An axiomatization of this type of algebras
(called Hom-Lie algebras) was given in \cite{JDS}, \cite{DS}. The associative counterpart of Hom-Lie algebras
(called Hom-associative algebras) has been introduced in \cite{ms}, where it was proved also
that the commutator bracket defined by the multiplication in a Hom-associative
algebra gives rise to a Hom-Lie algebra. Since then, Hom analogues of many other classical algebraic structures have been
introduced and studied (such as Hom-bialgebras, Hom-pre-Lie algebras, Hom-dendriform algebras etc.).

 A categorical approach to Hom-type algebras was considered  in \cite{stef}.  An attempt to generalize the construction
in \cite{stef} by including a group action led in \cite{gmmp} to the observation that Hom-structures have a very
natural generalization, in which a classical algebraic identity is twisted by two (commuting) homomorphisms,
instead of just one. This new type of algebras has been called BiHom-algebras (examples include
BiHom-associative algebras, BiHom-Lie algebras, BiHom-bialgebras etc.). The main tool to obtain examples of
Hom-algebras from classical algebras, the so-called ''Yau twisting'', works perfectly fine also in the BiHom-type  case.
There is a growing literature on Hom and BiHom-type algebras, let us just mention the very recent papers
\cite{Canad}, \cite{Colloq}, \cite{lmmp4} (see also references therein).

Novikov algebras appeared in connection with the Poisson brackets of hydrodynamic
type  (\cite{bn}, \cite{dn1}, \cite{dn2}) and Hamiltonian operators in the formal variational calculus (\cite{gd},
\cite{Xu1}). They are a special class of pre-Lie algebras and are related to several branches of geometry and
mathematical physics, such as Lie groups and algebras, affine manifolds, vertex and conformal algebras etc.

Novikov-Poisson algebras were introduced by Xu in \cite{Xu1}, motivated by the study of simple
Novikov algebras and irreducible modules. A Novikov-Poisson algebra is a Novikov algebra having an extra structure
(a commutative associative product) together with some compatibility conditions between the two structures.
As proved by Xu in \cite{Xu1}, Novikov-Poisson algebras have the remarkable property (not shared by Novikov algebras) that
they are closed under tensor products; also, Xu proved in \cite{Xu2} that they are also closed under certain types
of perturbations.

Hom-type analogues of Novikov and Novikov-Poisson algebras have been introduced and studied by Yau in \cite{yaunovikov},
\cite{yau3}. Examples can be obtained by Yau twisting from their classical counterparts. Also, it turns out that many properties
of Novikov and Novikov-Poisson algebras are shared by their Hom-type  counterparts; in particular, Yau proved that
Hom-Novikov-Poisson algebras are closed under tensor products and certain types of perturbations.

In our previous paper \cite{lmmp4}, we introduced and studied the concept of BiHom-Novikov algebra
(which is different from the one introduced in \cite{guo}), in relation also with the so-called infinitesimal BiHom-bialgebras.
In a short section of \cite{lmmp4}, we introduced as well the concept of BiHom-Novikov-Poisson algebra and
provided some classes of examples. The aim of the present paper is to continue the study of BiHom-Novikov-Poisson algebras.
More precisely, we show that BiHom-Novikov-Poisson algebras are closed under tensor products and under certain types of perturbations, extending thus the results mentioned above obtained by Xu in the classical case and by Yau in the Hom-type case.

The most tricky part in what we are doing is contained in Lemma \ref{lemma 3.1}, and we try now to describe its content.
The perturbations used by Xu in the classical case are essentially based on the following obvious fact: if
$(A, \mu )$ is a commutative associative algebra (with multiplication $\mu $ denoted by juxtaposition) and $a\in A$ is a fixed
element, and we define a new multiplication on $A$ by $x\diamond y=axy$, then $(A, \diamond )$ is also a
commutative associative algebra. The Hom-type  analogue of this result (Lemma 4.1 in Yau's paper \cite{yau3})
reads as follows: if $(A, \mu , \alpha )$ is a commutative Hom-associative algebra and $a\in A$ is an element such that
$\alpha ^2(a)=a$ and we define a new multiplication on $A$ by $x\diamond y=a(xy)$, then $(A, \diamond , \alpha ^2)$ is also a
commutative Hom-associative algebra. We have tried to extend this result (by generalizing the formula for
$\diamond $) to the BiHom-type case and we failed. Then we realized that actually the natural formula for $\diamond $
that has to be extended (both in the Hom-type  and BiHom-type case) is not $x\diamond y=axy$, but $x\diamond y=xay$ (which
is associative even if the original multiplication is noncommutative). Thus, our Lemma \ref{lemma 3.1} reads as follows: if
$(A, \mu , \alpha , \beta )$ is a BiHom-associative algebra and $a\in A$ is an element satisfying
$\a^2(a)=\b^2 (a)=a$ and we define a new operation on $A$ by
$x\diamond y=\alpha (x)(\alpha (a)y)$, then $(A, \diamond , \alpha^2 , \beta^2 )$ is also a BiHom-associative algebra;
if moreover
$(A, \mu , \alpha , \beta )$ is BiHom-commutative, then $(A, \diamond , \alpha^2 , \beta^2 )$ is also
BiHom-commutative. It turns out that Yau's result can be obtained as a particular case of this (see the comments after our
Lemma \ref{lemma 3.1}).

In the last section of the paper we investigate a special class of BiHom-Novikov-Poisson algebras, called
''left BiHom-associative''; in the case of bijective structure maps, these are exactly those BiHom-Novikov-Poisson algebras
that give rise to so-called BiHom-Poisson algebras 
(a concept we introduce in this paper, where we consider a left-handed Leibniz identity, we refer to \cite{adimi}
 for the concept dealing with a right-handed version)
see Theorem
\ref{theorem 4.5}. It turns out that left BiHom-associativity is preserved by Yau twisting, by tensor products and
by the perturbations mentioned above.


\section{Preliminaries} \label{sec1}
\setcounter{equation}{0}

We work over a base field $\Bbbk $. All
algebras, linear spaces etc. will be over $\Bbbk $; unadorned $\otimes $
means $\otimes_{\Bbbk}$. We denote by $_{\Bbbk }\mathcal{M}$ the category of linear spaces over $\Bbbk $.
Unless otherwise specified, the
algebras (associative or not) that will appear in what follows are
\emph{not} supposed to be unital, and a multiplication $\mu :A\otimes
A\rightarrow A$ on a linear space $A$ is denoted by juxtaposition: $\mu
(v\otimes v^{\prime })=vv^{\prime }$. For the composition of two maps $f$
and $g$, we write either $g\circ f$ or simply $gf$. For the identity
map on a linear space $A$ we use the notation $id_A$.

\begin{definition} (\cite{gmmp})
A BiHom-associative algebra is a 4-tuple $\left( A,\mu ,\alpha ,\beta \right) $, where $A$ is
a linear space and $\alpha , \beta :A\rightarrow A$
and $\mu :A\otimes A\rightarrow A$ are linear maps such that
$\alpha \circ \beta =\beta \circ \alpha $,
$\alpha (xy) =\alpha (x)\alpha (y)$, $\beta (xy)=\beta (x)\beta (y)$ and
\begin{eqnarray}
\alpha (x)(yz)=(xy)\beta (z), \;\;\;\forall \; x, y, z\in A. \label{BHassoc}
\end{eqnarray}
The maps $\alpha $ and $\beta $ (in this order) are called the structure maps
of $A$ and condition (\ref{BHassoc}) is called the BiHom-associativity condition.
A morphism $f:(A, \mu _A , \alpha _A, \beta _A)\rightarrow (B, \mu _B ,
\alpha _B, \beta _B)$ of BiHom-associative algebras is a linear map $%
f:A\rightarrow B$ such that $\alpha _B\circ f=f\circ \alpha _A$, $\beta
_B\circ f=f\circ \beta _A$ and $f\circ \mu_A=\mu _B\circ (f\otimes f)$.
\end{definition}

If $(A, \mu )$ is an associative algebra and $\alpha , \beta :A\rightarrow A$ are two commuting algebra maps, then
$A_{(\alpha , \beta )}:=(A, \mu \circ (\alpha \otimes \beta ), \alpha , \beta )$ is a BiHom-associative algebra,
called the Yau twist of $A$ given by the maps $\alpha $ and $\beta $. 

\begin{definition}(\cite{lmmp4})
\label{commbinovi} A BiHom-associative algebra $(A, \mu, \alpha , \beta)$ is
called BiHom-commutative if
\begin{eqnarray}
\beta(a)\alpha(b)=\beta(b)\alpha(a), \;\;\; \forall \;\;a, b\in A. \label{commu}
\label{commnovi}
\end{eqnarray}
\end{definition}
\begin{definition}
A left pre-Lie algebra is an algebra $(A, * )$ for which
\begin{eqnarray*}
&&x*(y*z)-(x*y)*z=y*(x*z)-(y*x)*z, \;\;\;  \forall \;\;x, y, z\in A.
\end{eqnarray*}

An algebra $(A, * )$ is called Novikov algebra if it is left pre-Lie and
\begin{eqnarray*}
&&(x*y)*z=(x*z)*y, \;\;\; \forall \;x, y, z\in A.
\end{eqnarray*}
A morphism of Novikov algebras from $(A, * )$ to $(A', * ')$ is a linear map
$f :A\rightarrow A'$ satisfying $f(x*y)=f(x)*'f(y)$, for all $x, y\in A$.
\end{definition}

\begin{definition} (\cite{Xu1}, \cite{Xu2}) A Novikov-Poisson algebra is a triple $(A, \mu , *)$ such that
$(A, \mu )$ is a  commutative associative algebra, $(A, * )$ is a Novikov algebra and
the following compatibility conditions hold, for all $x, y, z\in A$:
 \begin{eqnarray}
&& (x*y)z-x* (yz)=(y*x)z-y* (xz),~~~~ \label{NP1}\\
 && (xy)* z=(x*z)y. \label{NP2}
\end{eqnarray}

A morphism of Novikov-Poisson algebras from $(A, \mu , * )$ to $(A', \mu ', *')$ is a linear map
$f :A\rightarrow A'$ satisfying $f\circ \mu =\mu '\circ (f\otimes f)$ and $f(x*y)=f(x)*'f(y)$, for all $x, y\in A$.
\end{definition}

Note that, by the commutativity of $(A, \mu )$, (\ref{NP2}) is equivalent to
\begin{eqnarray}
&&(xy)*z=x(y*z), \;\;\; \forall \; x, y, z\in A. \label{NP3}
\end{eqnarray}

\begin{definition}(\cite{lmmp4})
\label{binovialgebra} A BiHom-Novikov algebra is a 4-tuple $(A, \ast, \alpha
, \beta)$, where $A$ is a linear space, $\ast: A \otimes A\rightarrow A$ is a
linear map and $%
\alpha , \beta : A \rightarrow A$ are commuting linear maps (called the structure maps of $A$),
satisfying the
following conditions, for all $x, y, z \in A$:
\begin{eqnarray}
&& \alpha(x\ast y)=\alpha(x)\ast \alpha(y),~~ \beta(x\ast y)=\beta(x)\ast \beta(y),
\label{BiNovik} \\
&&
(\beta(x)\ast \alpha(y))\ast \beta(z)-\alpha\beta(x)\ast (\alpha(y)\ast z) \nonumber \\
&&\;\;\;\;\;\;\;\;\;\;\;\;=(\beta(y)\ast \alpha(x))\ast
\beta(z)-\alpha\beta(y)\ast (\alpha(x)\ast z),    \label{BiNoviko} \\
&& (x \ast \beta (y))\ast \alpha\beta(z)=(x \ast \beta (z))\ast \alpha\beta(y).
\label{Binovikov}
\end{eqnarray}

A morphism $f:(A, \ast _A , \alpha _A, \beta _A)\rightarrow (B, \ast _B ,
\alpha _B, \beta _B)$ of BiHom-Novikov algebras is a linear map
$f:A\rightarrow B$ such that $\alpha _B\circ f=f\circ \alpha _A$, $\beta
_B\circ f=f\circ \beta _A$ and $f(x\ast _Ay)=f(x)\ast _Bf(y)$, for all $x, y\in A$.
\end{definition}

\begin{definition}\label{binovipoisson algebra}(\cite{lmmp4})
A BiHom-Novikov-Poisson algebra is a 5-tuple $(A, \mu , \ast, \alpha, \beta)$ such that:

(1) $(A, \mu ,  \alpha , \beta)$ is a  BiHom-commutative algebra;

(2) $(A, \ast,  \alpha , \beta)$ is a BiHom-Novikov algebra;

(3) the following compatibility conditions hold for all $x, y, z\in A$:
 \begin{eqnarray}
&& (\beta(x)\ast \alpha(y)) \beta(z)-\alpha\beta(x)\ast (\alpha(y)z)=(\beta(y)\ast \alpha(x))\beta(z)-\alpha\beta(y)\ast (\alpha(x)z),~~~~ \label{4.1}\\
 && (x \beta (y))\ast \alpha\beta(z)=(x\ast \beta (z))\alpha\beta(y),
\label{4.2}\\
&&\alpha (x)(y\ast z)=(xy)\ast \beta (z). \label{1.11}
\end{eqnarray}

The maps $\alpha $ and $\beta $ (in this order) are called the structure maps of $A$.

A morphism $f:(A, \mu , \ast, \alpha, \beta)\rightarrow (A', \mu ', \ast ', \alpha ', \beta ')$ of BiHom-Novikov-Poisson
algebras is a map that is a morphism of BiHom-associative algebras from $(A, \mu , \alpha , \beta )$ to
$(A', \mu ', \alpha ', \beta ')$ and a morphism of BiHom-Novikov algebras from $(A, \ast , \alpha , \beta )$ to
$(A', \ast ', \alpha ', \beta ')$.
\end{definition}

From Lemma 3.2 in \cite{lmmp4}, we know that if $\alpha $ and $\beta $ are bijective, the conditions (\ref{4.2})
and (\ref{1.11}) are equivalent (assuming (1) and multiplicativity of $\alpha $ and $\beta $ with respect to $\ast $).
\section{Tensor products} \label{sec2}
\setcounter{equation}{0}
Novikov algebras are not closed under tensor products in a non-trivial way. One reason Novikov-Poisson algebras
were introduced in \cite{Xu1}
was that the tensor product of two Novikov-Poisson algebras is a Novikov-Poisson algebra
non-trivially (\cite{Xu1}, Theorem 4.1). This fact was extended by Yau (see \cite{yau3}) to the class of Hom-Novikov-Poisson
algebras, and the aim of this section is to further extend it to
BiHom-Novikov-Poisson algebras.

\begin{theorem} \label {tensor}  Let $(A_i, \cdot_i, \ast_i, \alpha_i, \beta_i )$ be BiHom-Novikov-Poisson algebras, for $i=1, 2$. Then $A:=(A_1\otimes A_2, \cdot, \ast, \alpha_1 \otimes \alpha_2 ,
\beta_1\otimes \beta_2 )$ is a BiHom-Novikov-Poisson algebra, where the operations $\cdot $ and $\ast $ are
defined as follows (for $x_i, y_i \in A_i$):
\begin{eqnarray*}
(x_1\otimes x_2)\cdot (y_1\otimes y_2)&=&(x_1\cdot_1 y_1)\otimes (x_2\cdot_2 y_2),\\
(x_1\otimes x_2)\ast (y_1\otimes y_2)&=&(x_1\ast_1 y_1)\otimes (x_2\cdot_2 y_2)+(x_1\cdot_1 y_1)\otimes (x_2\ast_2 y_2).
\end{eqnarray*}
\end{theorem}

\begin{proof} To improve readability, we omit the subscripts in $\cdot_{i}$ and $\ast_{i}$ in the proof.

(1) We show that $(A_1\otimes A_2, \cdot, \alpha_1 \otimes \alpha_2 ,
\beta_1\otimes \beta_2 )$ is a BiHom-commutative algebra. For $x_1, y_1\in A_1$, $x_2, y_2 \in A_2$, it is easy to check that $(\a_1\ot \a_2)[(x_1\otimes x_2)\cdot (y_1\otimes y_2)]=[(\a_1\ot \a_2)(x_1\otimes x_2)]\cdot [(\a_1\ot \a_2)(y_1\otimes y_2)]$ and $(\b_1\ot \b_2)[(x_1\otimes x_2)\cdot (y_1\otimes y_2)]=[(\b_1\ot \b_2)(x_1\otimes x_2)]\cdot [(\b_1\ot \b_2)(y_1\otimes y_2)]$.

We verify the BiHom-associativity condition (\ref{BHassoc}):\\[2mm]
$[(\a_1 \ot \a_2)(x_1\otimes x_2)]\cdot [(y_1 \ot y_2)\cdot(z_1\otimes z_2)]$
\begin{eqnarray*}
&=& (\a_1(x_1)\otimes \a_2(x_2))\cdot[(y_1\cdot z_1)\otimes (y_2\cdot z_2)]
=(\a_1(x_1)\cdot (y_1\cdot z_1))\otimes (\a_2(x_2)\cdot (y_2\cdot z_2))\\
&\overset{(\ref{BHassoc})}{=}& ((x_1\cdot y_1)\cdot \b_1(z_1))\otimes ((x_2\cdot y_2)\cdot \b_2(z_2))
=[(x_1\cdot y_1)\otimes (x_2\cdot y_2)]\cdot (\b_1(z_1)\otimes \b_2(z_2))\\
&=& [(x_1\otimes x_2)\cdot (y_1\otimes y_2)]\cdot [(\b_1\ot \b_2)(z_1\otimes z_2)].
\end{eqnarray*}
Now we check the BiHom-commutativity condition (\ref{commu}):
\begin{eqnarray*}
[(\b_1 \ot \b_2)(x_1\otimes x_2)]\cdot [(\a_1 \ot \a_2)(y_1\otimes y_2)]
&=& (\b_1(x_1)\cdot \a_1(y_1))\otimes (\b_2(x_2)\cdot \a_2(y_2))\\
&\overset{(\ref{commu})}{=}& (\b_1(y_1)\cdot \a_1(x_1))\otimes (\b_2(y_2)\cdot \a_2(x_2))\\
&=& [(\b_1 \ot \b_2)(y_1\otimes y_2)]\cdot [(\a_1 \ot \a_2)(x_1\otimes x_2)].
\end{eqnarray*}
So indeed $A$ is a BiHom-commutative algebra.

(2) We prove that $(A_1\otimes A_2, \ast, \a=\alpha_1 \otimes \alpha_2 ,
\b=\beta_1\otimes \beta_2 )$ is a BiHom-Novikov algebra. It is obvious that (\ref{BiNovik}) holds. We prove now
 (\ref{Binovikov}). For $x=x_1\otimes x_2, y=y_1\otimes y_2$ and $z=z_1\otimes z_2$ in $A$, we have to show that $(x\ast \b(y))\ast \a \b(z)$ is symmetric in $y$ and $z$. We have:\\[2mm]
${\;\;\;}$
$(x\ast \b(y))\ast \a \b(z)$
\begin{eqnarray*}
&=& [(x_1\ast \b_1(y_1))\otimes (x_2\cdot \b_2(y_2))+(x_1\cdot \b_1(y_1))\otimes (x_2\ast \b_2(y_2))]\ast (\a_1 \b_1(z_1)\otimes \a_2\b_2(z_2))\\
&=& \underbrace{[(x_1\ast \b_1(y_1))\ast \a_1\b_1(z_1)]\otimes [(x_2\cdot \b_2(y_2))\cdot \a_2\b_2(z_2)]}_a \\
&& + \underbrace{[(x_1\ast \b_1(y_1))\cdot \a_1\b_1(z_1)]\otimes [(x_2\cdot \b_2(y_2))\ast \a_2\b_2(z_2)]}_{b(x, y, z)} \\
&& + \underbrace{[(x_1\cdot \b_1(y_1))\ast \a_1\b_1(z_1)]\otimes [(x_2\ast \b_2(y_2))\cdot \a_2\b_2(z_2)]}_{c(x, y, z)} \\
&& + \underbrace{[(x_1\cdot \b_1(y_1))\cdot \a_1\b_1(z_1)]\otimes [(x_2\ast \b_2(y_2))\ast \a_2\b_2(z_2)]}_{d}.
\end{eqnarray*}
 We prove that $a$ is symmetric in $y$ and $z$ and leave to prove the same thing about
$d$ to the reader:
\begin{eqnarray*}
a &=&[(x_1\ast \b_1(y_1))\ast \a_1\b_1(z_1)]\otimes [(x_2\cdot \b_2(y_2))\cdot \a_2\b_2(z_2)]\\
&\overset{(\ref{BHassoc}), (\ref{Binovikov})}{=}& [(x_1\ast \b_1(z_1))\ast \a_1\b_1(y_1)]
\otimes [\a_2(x_2)\cdot (\b_2(y_2)\cdot \a_2(z_2))]\\
&\overset{(\ref{commu})}{=}& [(x_1\ast \b_1(z_1))\ast \a_1\b_1(y_1)]\otimes [\a_2(x_2)\cdot (\b_2(z_2)\cdot \a_2(y_2))]\\
&\overset{(\ref{BHassoc})}{=}& [(x_1\ast \b_1(z_1))\ast \a_1\b_1(y_1)]\otimes [(x_2\cdot \b_2(z_2))\cdot \a_2\b_2(y_2)].
\end{eqnarray*}
Moreover, we have
\begin{eqnarray*}
b(x, y, z)&=& [(x_1\ast \b_1(y_1))\cdot \a_1\b_1(z_1)]\otimes [(x_2\cdot \b_2(y_2))\ast \a_2\b_2(z_2)]\\
&\overset{(\ref{4.2})}{=}& [(x_1\cdot \b_1(z_1))\ast \a_1\b_1(y_1)]\otimes [(x_2\ast \b_2(z_2))\cdot \a_2\b_2(y_2)]
=c(x, z, y).
\end{eqnarray*}
Thus
\begin{eqnarray*}
(x\ast \b(y))\ast \a \b(z)&=& a + b(x, y, z) + c(x, y, z) +d,\\
(x\ast \b(z))\ast \a \b(y)&=& a + b(x, z, y) + c(x, z, y) +d\\
&=& a + c(x, y, z) + b(x, y, z) +d
=(x\ast \b(y))\ast \a \b(z).
\end{eqnarray*}
This shows that $(x\ast \b(y))\ast \a \b(z)$ is symmetric in $y$ and $z$, so formula (\ref{Binovikov}) holds in $A$.

For (\ref{BiNoviko}) in $A$ we must show that
$(\b(x) \ast \a(y))\ast \b(z)- \a\b(x)\ast(\a(y)\ast z)$ with respect to $\ast$ in $A$ is symmetric in $x$ and $y$.
We compute:\\[2mm]
$(\b(x) \ast \a(y))\ast \b(z)- \a\b(x)\ast(\a(y)\ast z)$
\begin{eqnarray*}
&=& [(\b_1(x_1)\ast \a_1(y_1))\ast \b_1(z_1)]\otimes [(\b_2(x_2)\cdot \a_2(y_2))\cdot \b_2(z_2)]\\
&& +[(\b_1(x_1)\ast \a_1(y_1))\cdot \b_1(z_1)]\otimes [(\b_2(x_2)\cdot \a_2(y_2))\ast \b_2(z_2)]\\
&& + [(\b_1(x_1)\cdot \a_1(y_1))\ast \b_1(z_1)]\otimes [(\b_2(x_2)\ast \a_2(y_2))\cdot \b_2(z_2)]\\
&& + [(\b_1(x_1)\cdot \a_1(y_1))\cdot \b_1(z_1)]\otimes [(\b_2(x_2)\ast \a_2(y_2))\ast \b_2(z_2)]\\
&& -[\a_1\b_1(x_1)\ast (\a_1(y_1)\ast z_1)] \otimes [\a_2\b_2(x_2)\cdot (\a_2(y_2)\cdot z_2)]\\
&& -[\a_1\b_1(x_1)\cdot (\a_1(y_1)\ast z_1)] \otimes [\a_2\b_2(x_2)\ast (\a_2(y_2)\cdot z_2)]\\
&& -[\a_1\b_1(x_1)\ast (\a_1(y_1)\cdot z_1)] \otimes [\a_2\b_2(x_2)\cdot (\a_2(y_2)\ast z_2)]\\
&& -[\a_1\b_1(x_1)\cdot (\a_1(y_1)\cdot z_1)] \otimes [\a_2\b_2(x_2)\ast (\a_2(y_2)\ast z_2)]\\
&\overset{(\ref{BHassoc})}{\underset{(\ref{1.11})}{=}}& [(\b_1(x_1)\ast \a_1(y_1))\ast \b_1(z_1)]
\otimes [(\b_2(x_2)\cdot \a_2(y_2))\cdot \b_2(z_2)]\\
&& + [(\b_1(x_1)\ast \a_1(y_1))\cdot \b_1(z_1)]\otimes [(\b_2(x_2)\cdot \a_2(y_2))\ast \b_2(z_2)]\\
&& + [(\b_1(x_1)\cdot \a_1(y_1))\ast \b_1(z_1)]\otimes [(\b_2(x_2)\ast \a_2(y_2))\cdot \b_2(z_2)]\\
&& + [(\b_1(x_1)\cdot \a_1(y_1))\cdot \b_1(z_1)]\otimes [(\b_2(x_2)\ast \a_2(y_2))\ast \b_2(z_2)]\\
&& -[\a_1\b_1(x_1)\ast (\a_1(y_1)\ast z_1)] \otimes [(\b_2(x_2)\cdot \a_2(y_2))\cdot\b_2( z_2)]\\
&& -[(\b_1(x_1)\cdot \a_1(y_1))\ast \b_1(z_1)] \otimes [\a_2\b_2(x_2)\ast (\a_2(y_2)\cdot z_2)]\\
&& -[\a_1\b_1(x_1)\ast (\a_1(y_1)\cdot z_1)] \otimes [(\b_2(x_2)\cdot \a_2(y_2))\ast \b_2(z_2)]\\
&& -[(\b_1(x_1)\cdot \a_1(y_1))\cdot \b_1(z_1)] \otimes [\a_2\b_2(x_2)\ast (\a_2(y_2)\ast z_2)]\\
&=& [(\b_1(x_1)\ast \a_1(y_1))\ast \b_1(z_1)-\a_1\b_1(x_1)\ast (\a_1(y_1)\ast z_1)]\otimes [(\b_2(x_2)\cdot \a_2(y_2))\cdot \b_2(z_2)]\\
&& +[(\b_1(x_1)\cdot \a_1(y_1))\cdot \b_1(z_1)]\otimes [(\b_2(x_2)\ast \a_2(y_2))\ast \b_2(z_2)-\a_2\b_2(x_2)\ast (\a_2(y_2)\ast z_2)]\\
&& +[(\b_1(x_1)\ast \a_1(y_1))\cdot \b_1(z_1)-\a_1\b_1(x_1)\ast (\a_1(y_1)\cdot z_1)]
\otimes [(\b_2(x_2)\cdot \a_2(y_2))\ast \b_2(z_2)]\\
&& +[(\b_1(x_1)\cdot \a_1(y_1))\ast \b_1(z_1)]\otimes [(\b_2(x_2)\ast \a_2(y_2))\cdot \b_2(z_2)-\a_2\b_2(x_2)\ast (\a_2(y_2)\cdot z_2)]\\
&\overset{(\ref{commu}), (\ref{BiNoviko})}{\underset{(\ref{4.1})}{=}}& [(\b_1(y_1)\ast \a_1(x_1))\ast \b_1(z_1)-\a_1\b_1(y_1)\ast (\a_1(x_1)\ast z_1)]\otimes [(\b_2(y_2)\cdot \a_2(x_2))\cdot \b_2(z_2)]\\
&& +[(\b_1(y_1)\cdot \a_1(x_1))\cdot \b_1(z_1)]\otimes [(\b_2(y_2)\ast \a_2(x_2))\ast \b_2(z_2)-\a_2\b_2(y_2)\ast (\a_2(x_2)\ast z_2)]\\
&& +[(\b_1(y_1)\ast \a_1(x_1))\cdot \b_1(z_1)-\a_1\b_1(y_1)\ast (\a_1(x_1)\cdot z_1)]
\otimes [(\b_2(y_2)\cdot \a_2(x_2))\ast \b_2(z_2)]\\
&& +[(\b_1(y_1)\cdot \a_1(x_1))\ast \b_1(z_1)]\otimes [(\b_2(y_2)\ast \a_2(x_2))\cdot \b_2(z_2)-\a_2\b_2(y_2)\ast (\a_2(x_2)\cdot z_2)]\\
&=& (\b(y) \ast \a(x))\ast \b(z)- \a\b(y)\ast(\a(x)\ast z).
\end{eqnarray*}
(3) We prove the relations (\ref{4.1})-(\ref{1.11}). To prove (\ref{4.2}) in $A$, we compute:\\[2mm]
$(x\cdot \b(y))\ast \a \b(z)$
\begin{eqnarray*}
&=& [(x_1\cdot \b_1(y_1))\otimes (x_2\cdot \b_2(y_2))]\ast (\a_1\b_1(z_1)\otimes \a_2\b_2(z_2))\\
&=& [(x_1\cdot \b_1(y_1))\ast \a_1\b_1(z_1)]\otimes [(x_2\cdot \b_2(y_2))\cdot \a_2\b_2(z_2)]\\
&& +[(x_1\cdot \b_1(y_1))\cdot \a_1\b_1(z_1)]\otimes [(x_2\cdot \b_2(y_2))\ast \a_2\b_2(z_2)]\\
&\overset{(\ref{BHassoc})}{\underset{(\ref{4.2})}{=}}& [(x_1\ast \b_1(z_1))\cdot \a_1\b_1(y_1)]\otimes [\a_2(x_2)\cdot (\b_2(y_2)\cdot \a_2(z_2))]\\
&& +[\a_1(x_1)\cdot (\b_1(y_1)\cdot \a_1(z_1))]\otimes [(x_2\ast \b_2(z_2))\cdot \a_2\b_2(y_2)]\\
&\overset{(\ref{commu})}{=}& [(x_1\ast \b_1(z_1))\cdot \a_1\b_1(y_1)]\otimes [\a_2(x_2)\cdot (\b_2(z_2)\cdot \a_2(y_2))]\\
&& +[\a_1(x_1)\cdot (\b_1(z_1)\cdot \a_1(y_1))]\otimes [(x_2\ast \b_2(z_2))\cdot \a_2\b_2(y_2)]\\
&\overset{(\ref{BHassoc})}{=}& [(x_1\ast \b_1(z_1))\cdot \a_1\b_1(y_1)]\otimes [(x_2\cdot \b_2(z_2))\cdot \a_2\b_2(y_2)]\\
&& +[(x_1\cdot \b_1(z_1))\cdot \a_1\b_1(y_1)]\otimes [(x_2\ast \b_2(z_2))\cdot \a_2\b_2(y_2)]\\
&=& [(x_1\ast \b_1(z_1))\otimes (x_2\cdot \b_2(z_2)) + (x_1\cdot \b_1(z_1))\otimes (x_2\ast \b_2(z_2))]
\cdot (\a_1\b_1(y_1) \otimes \a_2\b_2(y_2))\\
&=& [(x_1\otimes x_2)\ast (\b_1(z_1)\otimes \b_2(z_2))]\cdot (\a_1\b_1(y_1) \otimes \a_2\b_2(y_2))=
(x\ast \b(z))\cdot \a \b(y).
\end{eqnarray*}
To prove (\ref{4.1}) in $A$, we compute: \\[2mm]
${\;\;\;}$
$(\beta(x)\ast \alpha(y))\cdot \beta(z)-\alpha\beta(x)\ast (\alpha(y)\cdot z)$
\begin{eqnarray*}
&=& [(\b_1(x_1)\cdot \a_1(y_1))\cdot \b_1(z_1)]\otimes [(\b_2(x_2)\ast \a_2(y_2))\cdot \b_2(z_2)]\\
&& +[(\b_1(x_1)\ast \a_1(y_1))\cdot \b_1(z_1)]\otimes [(\b_2(x_2)\cdot \a_2(y_2))\cdot \b_2(z_2)]\\
&& -[\a_1\b_1(x_1)\cdot (\a_1(y_1)\cdot z_1)] \otimes [\a_2\b_2(x_2)\ast (\a_2(y_2)\cdot z_2)]\\
&& -[\a_1\b_1(x_1)\ast (\a_1(y_1)\cdot z_1)] \otimes [\a_2\b_2(x_2)\cdot (\a_2(y_2)\cdot z_2)]\\
&\overset{(\ref{BHassoc})}{=}&[(\b_1(x_1)\cdot \a_1(y_1))\cdot \b_1(z_1)]\otimes [(\b_2(x_2)\ast \a_2(y_2))\cdot \b_2(z_2)-\a_2\b_2(x_2)\ast (\a_2(y_2)\cdot z_2)]\\
&& +[(\b_1(x_1)\ast \a_1(y_1))\cdot \b_1(z_1)-\a_1\b_1(x_1)\ast (\a_1(y_1)\cdot z_1)]\otimes [(\b_2(x_2)\cdot \a_2(y_2))\cdot \b_2(z_2)]\\
&\overset{(\ref{commu})}{\underset{(\ref{4.1})}{=}}& [(\b_1(y_1)\cdot \a_1(x_1))\cdot \b_1(z_1)]\otimes [(\b_2(y_2)\ast \a_2(x_2))\cdot \b_2(z_2)- \a_2\b_2(y_2)\ast (\a_2(x_2)\cdot z_2)]\\
&& +[(\b_1(y_1)\ast \a_1(x_1))\cdot \b_1(z_1)-\a_1\b_1(y_1)\ast (\a_1(x_1)\cdot z_1)]\otimes [(\b_2(y_2)\cdot \a_2(x_2))\cdot \b_2(z_2)]\\
&=& (\beta(y)\ast \alpha(x))\cdot \beta(z)-\alpha\beta(y)\ast (\alpha(x)\cdot z).
\end{eqnarray*}
Now we prove (\ref{1.11}) in $A$. We compute:
\begin{eqnarray*}
(x\cdot y)\ast \beta (z)&=&((x_1\otimes x_2)\cdot (y_1\otimes y_2))\ast (\beta _1(z_1)\otimes \beta _2(z_2))\\
&=&(x_1\cdot y_1\otimes x_2\cdot y_2)\ast (\beta _1(z_1)\otimes \beta _2(z_2))\\
&=&(x_1\cdot y_1)\ast \beta _1(z_1)\otimes (x_2\cdot y_2)\cdot \beta _2(z_2)
+(x_1\cdot y_1)\cdot \beta _1(z_1)\otimes (x_2\cdot y_2)\ast \beta _2(z_2)\\
&\overset{(\ref{BHassoc})}{\underset{(\ref{1.11})}{=}}&
\alpha _1(x_1)\cdot (y_1\ast z_1)\otimes \alpha _2(x_2)\cdot (y_2\cdot z_2)
+\alpha _1(x_1)\cdot (y_1\cdot z_1)\otimes \alpha _2(x_2)\cdot (y_2\ast z_2)\\
&=&(\alpha _1(x_1)\otimes \alpha _2(x_2))\cdot (y_1\ast z_1\otimes y_2\cdot z_2+y_1\cdot z_1\otimes y_2\ast z_2)\\
&=&(\alpha _1(x_1)\otimes \alpha _2(x_2))\cdot ((y_1\otimes y_2)\ast (z_1\otimes z_2))
=\alpha (x)\cdot (y\ast z).
\end{eqnarray*}

From the above, it follows that  $A:=(A_1\otimes A_2, \cdot, \ast, \alpha_1 \otimes \alpha_2 ,
\beta_1\otimes \beta_2 )$ is a BiHom-Novikov-Poisson algebra.
\end{proof}

By taking in Theorem \ref{tensor} $\alpha_i =\beta_i =id_A$ for $i=1, 2$, we recover the following result, which is Theorem 4.1 in
\cite{Xu1}.

\begin{corollary}
 Let $(A_i, \cdot_i, \ast_i)$ be Novikov-Poisson algebras for $i=1, 2$. Then $A:=(A_1\otimes A_2, \cdot, \ast)$ is a Novikov-Poisson algebra defined as follows (for $x_i, y_i \in A_i$):
\begin{eqnarray*}
(x_1\otimes x_2)\cdot (y_1\otimes y_2)&=&(x_1\cdot_1 y_1)\otimes (x_2\cdot_2 y_2),\\
(x_1\otimes x_2)\ast (y_1\otimes y_2)&=&(x_1\ast_1 y_1)\otimes (x_2\cdot_2 y_2)+(x_1\cdot_1 y_1)\otimes (x_2\ast_2 y_2).
\end{eqnarray*}
\end{corollary}

By taking in Theorem \ref{tensor} $\alpha_i =\beta_i $ for $i=1, 2$, we recover the Hom-version of our result,
which is Theorem 3.1 in
\cite{yau3}.

\begin{corollary}
Let $(A_i, \cdot_i, \ast_i, \alpha_i )$ be Hom-Novikov-Poisson algebras for $i=1, 2$. Then $A:=(A_1\otimes A_2, \cdot, \ast, \alpha_1 \otimes \alpha_2 )$ is a Hom-Novikov-Poisson algebra defined as follows (for $x_i, y_i \in A_i$):
\begin{eqnarray*}
(x_1\otimes x_2)\cdot (y_1\otimes y_2)&=&(x_1\cdot_1 y_1)\otimes (x_2\cdot_2 y_2),\\
(x_1\otimes x_2)\ast (y_1\otimes y_2)&=&(x_1\ast_1 y_1)\otimes (x_2\cdot_2 y_2)+(x_1\cdot_1 y_1)\otimes (x_2\ast_2 y_2).
\end{eqnarray*}
\end{corollary}

By using Theorem \ref{tensor},  we can construct new BiHom-Novikov-Poisson algebras.
For example, the next result is obtained by first using Proposition 3.6 in \cite{lmmp4} and
then Theorem \ref{tensor}.
\begin{corollary}
Let $(A_i, \cdot_i, \ast_i, \alpha_i, \beta_i )$ be BiHom-Novikov-Poisson algebras for $i=1, 2$ and $A=A_1\otimes A_2$. For integers $n, m\geq 0$, define the linear maps $\a, \b: A\rightarrow A$ and $\cdot, \ast: A\ot A\rightarrow A$ by:
\begin{eqnarray*}
\a=\a_1^{n+1}\otimes \a_2^{m+1}&,& \b=\b_1^{n+1}\ot \b_2^{m+1},\\
(x_1\otimes x_2)\cdot (y_1\otimes y_2)&=&(\a_1^{n}(x_1)\cdot_1 \b_1^{n}(y_1))\otimes (\a_2^{m}(x_2)\cdot_2 \b_2^{m}(y_2)),\\
(x_1\otimes x_2)\ast (y_1\otimes y_2)&=&(\a_1^{n}(x_1)\ast_1 \b_1^{n}(y_1))\otimes (\a_2^{m}(x_2)\cdot_2 \b_2^{m}(y_2))\\
&&+(\a_1^{n}(x_1)\cdot_1 \b_1^{n}(y_1))\otimes (\a_2^{m}(x_2)\ast_2 \b_2^{m}(y_2)),
\end{eqnarray*}
for $x_i, y_i \in A_i$.  Then $(A, \cdot, \ast, \alpha,
\beta)$ is a BiHom-Novikov-Poisson algebra.
\end{corollary}

\begin{proof}
Indeed, by applying  Proposition 3.6 in \cite{lmmp4}  for $\tilde{\alpha }:=\alpha _1^{n}$ and
$\tilde{\beta }:=\beta _1^{n}$ and respectively for $\tilde{\alpha }:=\alpha _2^{m}$ and
$\tilde{\beta }:=\beta _2^{m}$,  we obtain that
\begin{eqnarray*}
 && A_1^{n}:=(A_1, \cdot_1 \circ (\alpha_1^{n}\otimes \beta_1^{n}), \ast_1 \circ (\alpha_1^{n}\otimes \beta_1^{n}), \alpha_1^{n+1}, \beta_1^{n+1}),\\
  && A_2^{m}:=(A_2, \cdot_2 \circ (\alpha_2^{m}\otimes \beta_2^{m}), \ast_2 \circ (\alpha_2^{m}\otimes \beta_2^{m}), \alpha_2^{m+1}, \beta_2^{m+1})
\end{eqnarray*}
are BiHom-Novikov-Poisson algebras. From Theorem \ref{tensor}, it follows that the tensor product $(A, \cdot, \ast, \alpha,
\beta)$ of $ A_1^{n}$ and $A_2^{m}$
 is also a BiHom-Novikov-Poisson algebra, where $A=A_1\otimes A_2$.
\end{proof}

\section{Perturbations of BiHom-Novikov-Poisson algebras}\label{sec3}
\setcounter{equation}{0}

The main aim of this section is to show that BiHom-Novikov-Poisson algebra structures are preserved under certain
''perturbations''.
We begin with a result of independent interest.
\begin{lemma}\label{lemma 3.1}
Let $(A, \mu , \alpha , \beta )$ be a BiHom-associative algebra and let $a\in A$ satisfying
\begin{eqnarray}
&&\a^2(a)=\b^2 (a)=a. \label{suplim}
\end{eqnarray}
Define a new operation on $A$ by
\begin{eqnarray}
\diamond : A\otimes A\rightarrow A, ~~x\diamond y=\alpha (x)(\alpha (a)y), \label{newdiamond}
\end{eqnarray}
for all $x, y\in A$. Then $A'=(A, \diamond , \alpha^2 , \beta^2 )$ is also a BiHom-associative algebra. If moreover
$(A, \mu , \alpha , \beta )$ is BiHom-commutative, then $A'=(A, \diamond , \alpha^2 , \beta^2 )$ is also
BiHom-commutative.
\end{lemma}
\begin{proof} The fact that $\alpha ^2$ and $\beta ^2$ are multiplicative with respect to $\diamond $ follows by an easy
computation (using (\ref{suplim})) which is left to the reader. We prove now that
$(A, \diamond , \alpha^2 , \beta^2 )$ is BiHom-associative. For $x, y, z\in A$ we compute:
\begin{eqnarray*}
 (x\diamond y)\diamond \b^2(z)&=&(\alpha (x)(\alpha (a)y))\diamond \beta ^2(z)
=(\alpha ^2(x)(\alpha ^2(a)\alpha (y)))(\alpha (a)\beta ^2(z))\\
&\overset{(\ref{BHassoc})}{=}&((\alpha (x)\alpha ^2(a))\alpha \beta (y))(\alpha (a)\beta ^2(z))\\
&\overset{(\ref{suplim})}{=}&((\alpha (x)\alpha ^2(a))\alpha \beta (y))(\alpha \beta ^2(a)\beta ^2(z))\\
&=&((\alpha (x)\alpha ^2(a))\alpha \beta (y))\beta (\alpha \beta (a)\beta (z))\\
&\overset{(\ref{BHassoc}), (\ref{suplim})}{=}&(\alpha ^2(x)\alpha (a))(\alpha \beta (y)(\alpha \beta (a)\beta (z)))
=(\alpha ^2(x)\alpha (a))\beta (\alpha (y)(\alpha (a)z))\\
&\overset{(\ref{BHassoc})}{=}&\alpha ^3(x)(\alpha (a)(\alpha (y)(\alpha (a)z)))\\
&=&\alpha ^2(x)\diamond (\alpha (y)(\alpha (a)z))=\a^2(x)\diamond (y\diamond z).
\end{eqnarray*}

Now we assume that $(A, \mu , \alpha , \beta )$ is BiHom-commutative and we prove that
$(A, \diamond , \alpha^2 , \beta^2 )$ is also
BiHom-commutative:
\begin{eqnarray*}
\b^2(x)\diamond \a^2(y)&=& \alpha \beta ^2(x)(\alpha (a)\alpha ^2(y))
\overset{(\ref{BHassoc})}{=}(\beta ^2(x)\alpha (a))\alpha ^2\beta (y)\\
&\overset{(\ref{commu})}{=}& (\beta (a)\alpha \beta (x))\alpha ^2\beta (y)=\beta (a\alpha (x))\alpha (\alpha \beta (y))\\
&\overset{(\ref{commu})}{=}&\beta (\alpha \beta (y))\alpha (a\alpha (x))=\alpha (\beta ^2(y))(\alpha (a)\alpha ^2(x))
=\b^2(y)\diamond \a^2(x),
\end{eqnarray*}
finishing the proof.
\end{proof}

Let us see now what Lemma \ref{lemma 3.1} gives in the Hom-associative case.
Let $(A, \mu , \alpha )$ be a commutative Hom-associative algebra, i.e. $\alpha (x)(yz)=(xy)\alpha (z)$
and $xy=yx$, for all $x, y, z\in A$. Assume that $\alpha $ is multiplicative with respect to $\mu $
and let $a\in A$ such that $\alpha ^2(a)=a$. Then obviously $(A, \mu , \alpha , \alpha )$ is also a
BiHom-commutative algebra and we are in the hypotheses of Lemma \ref{lemma 3.1}; in this
particular case of Lemma \ref{lemma 3.1}, the multiplication (\ref{newdiamond}) may be rewritten as: 
\begin{eqnarray*}
x\diamond y&=&\alpha (x)(\alpha (a)y)=(x\alpha (a))\alpha (y)\\
&=&(\alpha (a)x)\alpha (y)=\alpha ^2(a)(xy)=a(xy).
\end{eqnarray*}
What we obtained, $x\diamond y=a(xy)$, is exactly formula (4.1.1)  in \cite{yau3}. Thus, Lemma \ref{lemma 3.1} is indeed a generalization of the
corresponding result (Lemma 4.1 in \cite{yau3}) for the multiplicative commutative Hom-associative case. However,
note that formula (4.1.1)  in \cite{yau3} gives a Hom-associative multiplication only in the commutative case;
for the noncommutative case, it is clear that the proper formula is (\ref{newdiamond}).

We can now use Lemma \ref{lemma 3.1} to prove the main result of this section.

\begin{theorem}
\label{theorem 1} Let $(A,\mu ,\ast ,\alpha ,\beta )$ be a
BiHom-Novikov-Poisson algebra and $a\in A$ with 
$\alpha ^{2}(a)=\beta ^{2}(a)=a$.
Then $A^{\prime }=(A,\diamond ,\ast _{\alpha , \beta },\alpha ^{2},\beta ^{2})$
is also a BiHom-Novikov-Poisson algebra, where
\begin{eqnarray*}
&&x\diamond y =\alpha (x)(\alpha (a)y), \;\;\;
x\ast _{\alpha , \beta }y =\alpha (x)\ast \beta (y), \;\;\; \forall \;x, y\in A.
\end{eqnarray*}
\end{theorem}

\begin{proof}
By Lemma \ref{lemma 3.1} we know that $(A,\diamond ,\alpha ^{2},\beta ^{2})$
is a BiHom-commutative algebra. By Corollary 3.9 in \cite{lmmp4} we get that
$(A,\ast _{\alpha , \beta },\alpha ^{2},\beta ^{2})$ is a BiHom-Novikov algebra
(the $n=1$ case). It remains to prove the compatibility conditions for $A'$. Let
$x,y,z\in A$, we prove (\ref{4.2}):\\
${\;\;\;}$
$(x\diamond \beta ^{2}(y))\ast _{\alpha , \beta }\alpha ^{2}\beta ^{2}(z)$
\begin{eqnarray*}
&=&(\alpha (x)(\alpha (a)\beta ^{2}(y)))\ast _{\alpha , \beta }\alpha
^{2}\beta ^{2}(z)
=\alpha ( \alpha (x)(\alpha (a)\beta ^{2}(y)))\ast \beta \alpha
^{2}\beta ^{2}(z) \\
&\overset{(\ref{suplim})}{=}& ( \alpha ^2 (x)
(a\alpha \beta ^{2}(y))) \ast \alpha \beta (\alpha \beta
^{2}(z))
\overset{(\ref{suplim})}{=}(\alpha ^{2}\left( x\right) \beta
(\beta \left( a\right) \alpha \beta (y)) )\ast \alpha \beta
(\alpha \beta ^{2}\left( z\right)) \\
&\overset{(\ref{4.2})}{=}&(\alpha ^{2}\left( x\right) \ast \beta
(\alpha \beta ^{2}\left( z\right) )) \alpha \beta
(\beta \left( a\right)  \alpha \beta (y))
=(\alpha ^{2}\left( x\right) \ast \alpha \beta ^{3}\left( z\right))
\left( \alpha \beta ^{2}\left( a\right)  \alpha ^{2}\beta
^{2}(y)\right) \\
&\overset{(\ref{suplim})}{=}&\alpha \left( \alpha \left( x\right)
\ast \beta ^{3}\left( z\right) \right)\left( \alpha \left(
a\right) \alpha ^{2}\beta ^{2}(y)\right)
=\left( \alpha \left( x\right) \ast \beta ^{3}\left( z\right) \right)
\diamond \alpha ^{2}\beta ^{2}(y) \\
&=&(x\ast _{\alpha , \beta }\beta ^{2}(z))\diamond \alpha ^{2}\beta ^{2}(y).
\end{eqnarray*}%

To prove the compatibility condition (\ref{4.1}) for $A^{\prime }$, we
compute:\\[2mm]
${\;\;\;\;\;}$
$(\beta ^{2}(x)\ast _{\alpha , \beta }\alpha ^{2}(y))\diamond \beta
^{2}(z)-\alpha ^{2}\beta ^{2}(x)\ast _{\alpha , \beta }(\alpha ^{2}(y)\diamond
z)$
\begin{eqnarray*}
&=& \left( \alpha \beta ^{2}(x)\ast \beta \alpha ^{2}(y)\right)
\diamond \beta ^{2}(z)-\alpha ^{3}\beta ^{2}(x)\ast \beta \left( \alpha
^{2}(y)\diamond z\right)  \\
&=&\alpha \left( \alpha \beta ^{2}(x)\ast \beta \alpha ^{2}(y)\right) \left(
\alpha \left( a\right) \beta ^{2}(z)\right) -\alpha ^{3}\beta ^{2}(x)\ast
\beta \left( \alpha ^{3}(y)\left( \alpha \left( a\right)  z\right)\right)  \\
&=&\left( \alpha ^{2}\beta ^{2}(x)\ast \alpha \beta \alpha ^{2}(y)\right)
\left( \alpha \left( a\right) \beta ^{2}(z)\right) -\alpha ^{3}\beta
^{2}(x)\ast \left( \beta \alpha ^{3}(y)\left( \beta \alpha \left( a\right)
 \beta \left( z\right)\right) \right)  \\
&\overset{(\ref{suplim})}{=}&(\beta (\alpha ^{2}\beta (x)) \ast
\alpha ( \beta \alpha ^{2}(y))) \beta (\alpha \beta \left(
a\right) \beta (z)) \\
&&-\alpha \beta (\alpha ^{2}\beta (x)
)\ast (\alpha (\beta \alpha ^{2}(y) )(\alpha \beta
\left( a\right) \beta (z))) \\
&\overset{(\ref{4.1})}{=}&(\beta (\beta \alpha ^{2}(y)) \ast
\alpha ( \alpha ^{2}\beta (x)) )\beta ( \alpha \beta \left(
a\right) \beta (z)) \\
&&-\alpha \beta (\beta \alpha ^{2}(y))
\ast (\alpha (\alpha ^{2}\beta (x)) ( \alpha \beta \left(
a\right) \beta (z)) ) \\
&\overset{(\ref{suplim})}{=}&\left( \alpha ^{2}\beta ^{2}(y)\ast \beta
\alpha ^{3}(x)\right) \left( \alpha \left( a\right) \beta ^{2}(z)\right)
-\alpha ^{3}\beta ^{2}(y)\ast \left( \alpha ^{3}\beta (x)\left( \beta \alpha
\left( a\right) \beta \left( z\right) \right) \right)  \\
&=&\alpha \left( \alpha \beta ^{2}(y)\ast \beta \alpha ^{2}(x)\right) \left(
\alpha \left( a\right) \beta ^{2}(z)\right) -\alpha ^{3}\beta ^{2}(y)\ast
\beta \left( \alpha ^{3}(x)\left( \alpha \left( a\right) z\right) \right)  \\
&=&\left( \alpha \beta ^{2}(y)\ast \beta \alpha ^{2}(x)\right) \diamond
\beta ^{2}(z)-\alpha ^{3}\beta ^{2}(y)\ast \beta (\alpha ^{2}(x)\diamond z)
\\
&=&(\beta ^{2}(y)\ast _{\alpha , \beta }\alpha ^{2}(x))\diamond \beta
^{2}(z)-\alpha ^{2}\beta ^{2}(y)\ast _{\alpha , \beta }(\alpha ^{2}(x)\diamond
z).
\end{eqnarray*}%
Finally, we prove (\ref{1.11}):
\begin{eqnarray*}
\alpha ^{2}(x)\diamond (y\ast _{\alpha , \beta }z) &=&\alpha ^{2}(x)\diamond
\left( \alpha \left( y\right) \ast \beta \left( z\right) \right)
=\alpha (\alpha ^{2}(x)) (\alpha (a)( \alpha \left(
y\right) \ast \beta \left( z\right) ) ) \\
&\overset{(\ref{BHassoc})}{=}&
(\alpha ^{2}(x) \alpha (a))\beta (\alpha
\left( y\right) \ast \beta \left( z\right)) =\alpha (\alpha
\left( x\right) a) (\beta \alpha \left( y\right) \ast \beta
^{2}\left( z\right) )  \\
&\overset{\left( \ref{1.11}\right) }{=}&\left( (\alpha \left( x\right)
a) \beta \alpha \left( y\right) \right) \ast \beta ^{3}\left( z\right)
\overset{\left( \ref{1.11}\right) }{=}
( \alpha ^{2}(x)\left( a\alpha \left( y\right) \right)) \ast
\beta ^{3}\left( z\right)  \\
&\overset{(\ref{suplim})}{=}&(\alpha ^{2}(x)(\alpha ^{2}(a)\alpha
\left( y\right) ))\ast \beta ^{3}\left( z\right)
=\alpha (\alpha (x)(\alpha (a)y)) \ast \beta ^{3}\left(
z\right) \\
&=&\alpha \left( x\diamond y\right) \ast \beta ^{3}\left( z\right)
=(x\diamond y)\ast _{\alpha ,\beta }\beta ^{2}(z),
\end{eqnarray*}
finishing the proof.
\end{proof}

The next result is the Hom-version of Theorem \ref{theorem 1} with $\a=\b$, which is Theorem 4.2 in
\cite{yau3}.
\begin{corollary}\label{Coro 3.3}
 Let $(A, \mu , \ast, \alpha)$ be a multiplicative Hom-Novikov-Poisson algebra and $a\in A$ an element satisfying $\a^2(a)=a$.
Then $A'=(A, \diamond, \ast_{\a}, \alpha^2)$ is also a multiplicative Hom-Novikov-Poisson algebra, where
$x\diamond y=a(xy)$ and
$x\ast_{\a}y =\a(x)\ast \a (y)$,
for all $x, y\in A$.
\end{corollary}

The next result is similar to Theorem \ref{theorem 1}; it uses, as in  Theorem \ref{theorem 1}, a suitable element in a
 BiHom-Novikov-Poisson algebra, but this time to perturb the other multiplication.
\begin{theorem}\label{theorem 2}
 Let $(A, \mu , \ast, \alpha , \beta )$ be a BiHom-Novikov-Poisson algebra and $a\in A$ with $\a^2(a)=\b^2 (a)=a$.
Then $\overline{A}=(A, \cdot_{\a , \b}, \times, \alpha^2, \beta^2)$ is also a BiHom-Novikov-Poisson algebra, where
\begin{eqnarray*}
&&x\cdot_{\a , \b} y = \a(x) \b(y),\;\;\;
x\times y = \a(x)\ast \b(y)+ \alpha (x)(\alpha (a)y), \;\;\; \forall \;x, y\in A.
\end{eqnarray*}
\end{theorem}
\begin{proof}
 By Corollary 3.8 in \cite{lmmp4}, we get that $(A, \cdot_{\a , \b}, \alpha^2, \beta^2)$ is a BiHom-commutative algebra
(the $n=1$ case). We show that $(A, \times, \alpha^2, \beta^2)$ is a BiHom-Novikov algebra in Lemma \ref{lemma 3.3} below.
The relations (\ref{4.1}), (\ref{4.2}) and (\ref{1.11})
for $\overline{A}$ are proved in Lemma \ref{lemma 3.2} below.
\end{proof}

\begin{lemma} \label{lemma 3.3}
In the hypotheses of Theorem \ref{theorem 2},
$(A, \times, \alpha^2, \beta^2)$ is a BiHom-Novikov algebra.
\end{lemma}
\begin{proof}
We prove first  condition (\ref{BiNovik}):
\begin{eqnarray*}
\a^2(x\times y)&=& \a^2(\a(x)\ast \b(y)+ \alpha (x)(\alpha (a) y))
\overset{(\ref{suplim})}{=}\a^3 (x)\ast \a^2\b(y)+ \a^3(x)  (\alpha (a) \a^2(y))\\
&=& \a(\a^2 (x))\ast \b(\a^2(y))+ \alpha (\alpha ^2(x))  (\alpha (a) \a^2(y))=\a^2(x)\times \a^2(y).
\end{eqnarray*}
Similarly one can prove that $\b^2(x\times y)=\b^2(x)\times \b^2(y)$.

To check  condition (\ref{Binovikov}), we compute:
\begin{eqnarray*}
(x\times \b^2(y))\times \a^2\b^2(z)
&=&(\alpha (x)\ast \beta ^3(y)+\alpha (x)(\alpha (a)\beta ^2(y)))\times \alpha ^2\beta ^2(z)\\
&\overset{(\ref{suplim})}{=}&
\underbrace{(\a^2(x)\ast \a\b^3(y))\ast \a ^2\b^3(z)}_{u}+
\underbrace{(\a^2(x)(a \a\b^2(y)))\ast \a^2\b^3(z)}_{v(x, y, z)}\\
&& + \underbrace{(\a^2(x)\ast \a\b^3(y))(\alpha (a)\a^2\b^2(z))}_{s(x, y, z)}\\
&&+
\underbrace{(\alpha ^2(x)(a\a\b^2(y)))(\alpha (a)\a^2\b^2(z))}_{t}.
\end{eqnarray*}
The term $u=(\a^2(x)\ast \b(\a\b^2(y)))\ast \a\b(\a\b^2(z))$ is symmetric in $y$ and $z$ by (\ref{Binovikov}) in
$(A, \ast, \alpha , \beta )$.

Now we compute:
\begin{eqnarray*}
v(x, y, z)&\overset{(\ref{suplim})}{=}&
(\a^2(x) (\b^2(a)\alpha \beta ^2(y)))\ast \a^2\b^3(z)
=(\a^2(x)\beta (\beta (a)\alpha \beta (y)))\ast \alpha \beta (\alpha \beta ^2(z))\\
&\overset{(\ref{4.2})}{=}& (\a^2(x)\ast \a\b^3(z))\alpha \beta (\beta (a)\alpha \beta (y))\\
&\overset{(\ref{suplim})}{=}&(\a^2(x)\ast \a\b^3(z))(\alpha (a)\alpha ^2 \beta ^2(y))
=s(x, z, y).
\end{eqnarray*}
Now we prove that $t$ is symmetric in $y$ and $z$:
\begin{eqnarray*}
t&\overset{(\ref{suplim})}{=}&(\alpha ^2(x)(a\a\b^2(y)))(\alpha \beta ^2(a)\a^2\b^2(z))
=(\alpha ^2(x)(a\a\b^2(y)))\beta (\alpha \beta (a)\a^2\b(z))\\
&\overset{(\ref{BHassoc})}{=}& \alpha ^3(x)((a\alpha \beta ^2(y))(\alpha \beta (a)\alpha ^2\beta (z)))
\overset{(\ref{suplim})}{=}\alpha ^3(x)((\beta ^2(a)\alpha \beta ^2(y))(\alpha \beta (a)\alpha ^2\beta (z)))\\
&=& \alpha ^3(x)(\beta (\beta (a)\alpha \beta (y))\alpha (\beta (a)\alpha \beta (z)))
\overset{(\ref{commu})}{=} \alpha ^3(x)(\beta (\beta (a)\alpha \beta (z))\alpha (\beta (a)\alpha \beta (y)))\\
&=&\alpha ^3(x)((\beta ^2(a)\alpha \beta ^2(z))(\alpha \beta (a)\alpha ^2\beta (y))),
\end{eqnarray*}
and the last expression coincides with the one obtained three steps before it with $y$ and $z$ interchanged,
proving that indeed $t$ is symmetric in $y$ and $z$. \\
From the above we get that $(x\times \b^2(y))\times \a^2\b^2(z)$ is symmetric in $y$ and $z$,
proving (\ref{Binovikov}).\\
Now we prove (\ref{BiNoviko}). We need to check that the following expression is symmetric in $x$ and $y$:\\
${\;\;\;}$
$(\b^2(x)\times \a^2(y))\times \b^2(z)- \a^2 \b^2(x)\times(\a^2(y)\times z)$
\begin{eqnarray*}
&=& (\a\b^2(x)\ast \a^2\b(y))\times \b^2(z) + (\a\b^2(x)(\alpha (a)\a^2(y)))\times \b^2(z)\\
&& - \a^2 \b^2(x)\times(\a^3(y)\ast \b(z))- \a^2 \b^2(x)\times (\alpha ^3(y)(\alpha (a)z))\\
&\overset{(\ref{suplim})}{=}&(\alpha ^2\beta ^2(x)\ast \alpha ^3\beta (y))\ast \beta ^3(z)+
(\alpha ^2\beta ^2(x)\ast \alpha ^3\beta (y))(\alpha (a)\beta ^2(z))\\
&&+(\alpha ^2\beta ^2(x)(a\alpha ^3(y)))\ast \beta ^3(z)+(\alpha ^2\beta ^2(x)(a\alpha ^3(y)))(\alpha (a)\beta ^2(z))\\
&&-\alpha ^3\beta ^2(x)\ast (\alpha ^3\beta (y)\ast \beta ^2(z))-\alpha ^3\beta ^2(x)(\alpha (a)(\alpha ^3(y)
\ast \beta (z)))\\
&&-\alpha ^3\beta ^2(x)\ast (\alpha ^3\beta (y)(\alpha \beta (a)\beta (z)))-
\alpha ^3\beta ^2(x)(\alpha (a)(\alpha ^3(y)(\alpha (a)z))).
\end{eqnarray*}
We denote the eight terms in this expression respectively by $t$, $v$, $u$, $w$, $t'$, $u'$,
$v'$, $w'$, so the expression reads
\begin{eqnarray*}
&t + v+ u+ w - t'- u'- v'- w'.
\end{eqnarray*}
We prove that $t-t'$ is symmetric in $x$ and $y$ and $w-w'=0=u-u'$.
We compute:
\begin{eqnarray*}
t-t'&=& (\a^2\b^2(x)\ast \a^3\b(y))\ast \b^3(z)- \a^3\b^2(x)\ast (\a^3\b(y)\ast \b^2(z))\\
&=& (\b(\a^2\b(x))\ast \a(\a^2\b(y)))\ast \b(\b^2(z))- \a\b(\a^2\b(x))\ast (\a(\a^2\b(y))\ast \b^2(z))\\
&\overset{(\ref{BiNoviko})}{=}& (\b(\a^2\b(y))\ast \a(\a^2\b(x)))\ast \b(\b^2(z))- \a\b(\a^2\b(y))\ast (\a(\a^2\b(x))\ast \b^2(z))\\
&=& (\a^2\b^2(y)\ast \a^3\b(x))\ast \b^3(z)- \a^3\b^2(y)\ast (\a^3\b(x)\ast \b^2(z)).
\end{eqnarray*}
So clearly $t-t'$ is symmetric in $x$ and $y$.
Next we compute:
\begin{eqnarray*}
w-w'&=&(\alpha ^2\beta ^2(x)(a\alpha ^3(y)))(\alpha (a)\beta ^2(z))-
\alpha ^3\beta ^2(x)(\alpha (a)(\alpha ^3(y)(\alpha (a)z)))\\
&\overset{(\ref{BHassoc})}{=}&(\alpha ^2\beta ^2(x)(a\alpha ^3(y)))(\alpha (a)\beta ^2(z))-
\alpha ^3\beta ^2(x)((a\alpha ^3(y))(\alpha\beta  (a)\beta (z)))\\
&\overset{(\ref{BHassoc})}{=}&(\alpha ^2\beta ^2(x)(a\alpha ^3(y)))(\alpha (a)\beta ^2(z))-
(\alpha ^2\beta ^2(x)(a\alpha ^3(y)))(\alpha\beta ^2 (a)\beta ^2(z)))
\overset{(\ref{suplim})}{=}0,
\end{eqnarray*}
\begin{eqnarray*}
u-u'&=& (\alpha ^2\beta ^2(x)(a\alpha ^3(y)))\ast \beta ^3(z)-\alpha ^3\beta ^2(x)(\alpha (a)(\alpha ^3(y)\ast \beta (z)))\\
&\overset{(\ref{1.11})}{=}&\alpha ^3\beta ^2(x)((a\alpha ^3(y))\ast \beta ^2(z))-
\alpha ^3\beta ^2(x)(\alpha (a)(\alpha ^3(y)\ast \beta (z)))\\
&\overset{(\ref{1.11})}{=}&\alpha ^3\beta ^2(x)(\alpha (a)(\alpha ^3(y)\ast \beta (z)))-
\alpha ^3\beta ^2(x)(\alpha (a)(\alpha ^3(y)\ast \beta (z)))=0.
\end{eqnarray*}
The last thing to prove is that $v-v'$ is symmetric in $x$ and $y$, that is, we need to prove that \\[2mm]
${\;\;\;\;\;\;\;\;\;\;\;}$
$(\a^2\b^2(x)\ast \a^3\beta (y)) (\alpha (a)\b^2(z))- \a^3\b^2(x)\ast (\alpha ^3\b(y)(\alpha \beta (a)\beta (z)))$
\begin{eqnarray*}
&=& (\a^2\b^2(y)\ast \a^3\beta (x))(\alpha (a) \b^2(z))- \a^3\b^2(y)\ast (\alpha ^3\b(x)(\a\beta (a)\b(z))).
\end{eqnarray*}
We prove this identity as follows: \\[2mm]
${\;\;\;}$
$(\a^2\b^2(x)\ast \a^3\beta (y)) (\alpha (a)\b^2(z))- (\a^2\b^2(y)\ast \a^3\beta (x))(\alpha (a) \b^2(z))$
\begin{eqnarray*}
&\overset{(\ref{suplim})}{=}&(\a^2\b^2(x)\ast \a^3\beta (y)) (\alpha \beta ^2(a)\b^2(z))- (\a^2\b^2(y)\ast \a^3\beta (x))(\alpha \beta ^2(a) \b^2(z))\\
&=&(\beta (\a^2\b(x))\ast \alpha (\a^2\beta (y))) \beta (\alpha \beta (a)\b(z))-
(\beta (\a^2\b(y))\ast \alpha (\a^2\beta (x)))\beta (\alpha \beta (a) \b(z))\\
&\overset{(\ref{4.1})}{=}&\alpha \beta (\alpha ^2\beta (x))\ast (\alpha (\alpha ^2\beta (y))(\alpha \beta (a)\beta (z)))-
\alpha \beta (\alpha ^2\beta (y))\ast (\alpha (\alpha ^2\beta (x))(\alpha \beta (a)\beta (z)))\\
&=&\a^3\b^2(x)\ast (\alpha ^3\b(y)(\alpha \beta (a)\beta (z)))-
\a^3\b^2(y)\ast (\alpha ^3\b(x)(\a\beta (a)\b(z))).
\end{eqnarray*}

From all the above, it follows that indeed the expression
 $(\b^2(x)\times \a^2(y))\times \b^2(z)- \a^2 \b^2(x)\times(\a^2(y)\times z)$ is symmetric in $x$ and $y$,
finishing the proof.
\end{proof}
\begin{lemma}\label{lemma 3.2}
In the hypotheses of Theorem \ref{theorem 2}, $\overline{A}$ satisfies the compatibility conditions (\ref{4.1}), (\ref{4.2})
and (\ref{1.11}).
\end{lemma}
\begin{proof}
To prove  (\ref{4.2}) for $\overline{A}$, we compute as follows:\\[2mm]
${\;\;\;\;\;}$
$ (x\cdot_{\a, \b} \b^2(y))\times \a^2\b^2(z)$
\begin{eqnarray*}
&=& (\alpha (x)\beta ^3(y))\times \alpha ^2\beta ^2(z)\\
&=&(\alpha ^2(x)\alpha \beta ^3(y))\ast \alpha ^2\beta ^3(z)+(\alpha ^2(x)\alpha \beta ^3(y))
(\alpha (a)\alpha ^2\beta ^2(z))\\
&=&(\alpha ^2(x)\beta (\alpha \beta ^2(y)))\ast \alpha \beta (\alpha \beta ^2(z))+
\alpha (\alpha (x)\beta ^3(y))(\alpha (a)\alpha ^2\beta ^2(z))\\
&\overset{(\ref{BHassoc}), (\ref{4.2})}{=}&(\alpha ^2(x)\ast \beta (\alpha \beta ^2(z)))\alpha \beta (\alpha \beta ^2(y))+
((\alpha (x)\beta ^3(y))\alpha (a))\alpha ^2\beta ^3(z)\\
&\overset{(\ref{suplim})}{=}&(\alpha ^2(x)\ast \alpha \beta ^3(z))\alpha ^2 \beta ^3(y)+
((\alpha (x)\beta ^3(y))\alpha \beta ^2(a))\alpha ^2\beta ^3(z)\\
&\overset{(\ref{BHassoc})}{=}&(\alpha ^2(x)\ast \alpha \beta ^3(z))\alpha ^2 \beta ^3(y)+
(\alpha ^2(x)(\beta ^3(y)\alpha \beta (a)))\alpha ^2\beta ^3(z)\\
&\overset{(\ref{commu})}{=}&(\alpha ^2(x)\ast \alpha \beta ^3(z))\alpha ^2 \beta ^3(y)+
(\alpha ^2(x)(\beta ^2(a)\alpha \beta ^2(y)))\alpha ^2\beta ^3(z)\\
&\overset{(\ref{BHassoc}), (\ref{suplim})}{=}&(\alpha ^2(x)\ast \alpha \beta ^3(z))\alpha ^2 \beta ^3(y)+
((\alpha (x)a)\alpha \beta ^3(y))\alpha ^2\beta ^3(z)\\
&\overset{(\ref{BHassoc})}{=}&(\alpha ^2(x)\ast \alpha \beta ^3(z))\alpha ^2 \beta ^3(y)+
(\alpha ^2(x)\alpha (a))(\alpha \beta ^3(y)\alpha ^2\beta ^2(z))\\
&\overset{(\ref{commu})}{=}&(\alpha ^2(x)\ast \alpha \beta ^3(z))\alpha ^2 \beta ^3(y)+
(\alpha ^2(x)\alpha (a))(\alpha \beta ^3(z)\alpha ^2\beta ^2(y))\\
&\overset{(\ref{BHassoc})}{=}&(\alpha ^2(x)\ast \alpha \beta ^3(z))\alpha ^2 \beta ^3(y)+
((\alpha (x)a)\alpha \beta ^3(z))\alpha ^2\beta ^3(y)\\
&\overset{(\ref{BHassoc})}{=}&(\alpha ^2(x)\ast \alpha \beta ^3(z))\alpha ^2 \beta ^3(y)+
(\alpha ^2(x)(a\alpha \beta ^2(z)))\alpha ^2\beta ^3(y)\\
&\overset{(\ref{suplim})}{=}&(\alpha (x)\ast \beta ^3(z)+\alpha (x)(\alpha (a)\beta ^2(z)))\cdot _{\alpha , \beta }
\alpha ^2\beta ^2(y)
=(x\times \beta ^2(z))\cdot _{\alpha , \beta }
\alpha ^2\beta ^2(y).
\end{eqnarray*}

To prove (\ref{4.1}) for $\overline{A}$, we compute:\\[2mm]
${\;\;\;}$
$(\b^2(x)\times \a^2(y))\cdot_{\a , \b} \b^2(z)- \a^2\b^2(x) \times (\a^2(y)\cdot_{\a , \b} z)$
\begin{eqnarray*}
&=&(\alpha \beta ^2(x)\ast \alpha ^2\beta (y)+\alpha \beta ^2(x)(\alpha (a)\alpha ^2(y)))
\cdot_{\a , \b} \b^2(z)\\
&&- \a^2\b^2(x) \times (\a^3(y)\beta (z))\\
&\overset{(\ref{suplim})}{=}&(\alpha ^2\beta ^2(x)\ast \alpha ^3\beta (y))\beta ^3(z)+(\alpha ^2\beta ^2(x)(a\alpha ^3(y)))\beta ^3(z)\\
&&-\alpha ^3\beta ^2(x)\ast (\alpha ^3\beta (y)\beta ^2(z))-\alpha ^3\beta ^2(x)(\alpha (a)(\alpha ^3(y)\beta (z)))\\
&\overset{(\ref{BHassoc}), (\ref{suplim})}{=}&(\beta (\alpha ^2\beta (x))\ast \alpha (\alpha ^2\beta (y)))\beta (\beta ^2(z))
-\alpha \beta (\alpha ^2\beta (x))\ast (\alpha (\alpha ^2\beta (y))\beta ^2(z))\\
&&+(\alpha ^2\beta ^2(x)(\beta ^2(a)\alpha ^3(y)))\beta ^3(z)-
\alpha ^3\beta ^2(x)((a\alpha ^3(y))\beta ^2(z))\\
&\overset{(\ref{commu}), (\ref{suplim})}{=}&(\beta (\alpha ^2\beta (x))\ast \alpha (\alpha ^2\beta (y)))\beta (\beta ^2(z))
-\alpha \beta (\alpha ^2\beta (x))\ast (\alpha (\alpha ^2\beta (y))\beta ^2(z))\\
&&+(\alpha ^2\beta ^2(x)(\alpha ^2\beta (y)\alpha \beta (a)))\beta ^3(z)-
\alpha ^3\beta ^2(x)((\beta ^2(a)\alpha ^3(y))\beta ^2(z))\\
&\overset{(\ref{commu}), (\ref{suplim}), (\ref{BHassoc})}{=}&
(\beta (\alpha ^2\beta (x))\ast \alpha (\alpha ^2\beta (y)))\beta (\beta ^2(z))
-\alpha \beta (\alpha ^2\beta (x))\ast (\alpha (\alpha ^2\beta (y))\beta ^2(z))\\
&&+((\alpha \beta ^2(x)\alpha ^2\beta (y))\alpha (a))\beta ^3(z)-
\alpha ^3\beta ^2(x)((\alpha ^2\beta (y)\alpha \beta (a))\beta ^2(z))\\
&\overset{(\ref{commu}), (\ref{BHassoc})}{=}&
(\beta (\alpha ^2\beta (x))\ast \alpha (\alpha ^2\beta (y)))\beta (\beta ^2(z))
-\alpha \beta (\alpha ^2\beta (x))\ast (\alpha (\alpha ^2\beta (y))\beta ^2(z))\\
&&+((\alpha \beta ^2(y)\alpha ^2\beta (x))\alpha (a))\beta ^3(z)-
\alpha ^3\beta ^2(x)(\alpha ^3\beta (y)(\alpha \beta (a)\beta (z)))\\
&\overset{(\ref{suplim}), (\ref{BHassoc})}{=}&
(\beta (\alpha ^2\beta (x))\ast \alpha (\alpha ^2\beta (y)))\beta (\beta ^2(z))
-\alpha \beta (\alpha ^2\beta (x))\ast (\alpha (\alpha ^2\beta (y))\beta ^2(z))\\
&&+((\alpha \beta ^2(y)\alpha ^2\beta (x))\alpha (a))\beta ^3(z)-
(\alpha ^2\beta ^2(x)\alpha ^3\beta (y))(\alpha (a)\beta ^2(z))\\
&\overset{(\ref{suplim}), (\ref{commu})}{=}&
(\beta (\alpha ^2\beta (x))\ast \alpha (\alpha ^2\beta (y)))\beta (\beta ^2(z))
-\alpha \beta (\alpha ^2\beta (x))\ast (\alpha (\alpha ^2\beta (y))\beta ^2(z))\\
&&+((\alpha \beta ^2(y)\alpha ^2\beta (x))\alpha \beta ^2(a))\beta ^3(z)-
(\alpha ^2\beta ^2(y)\alpha ^3\beta (x))(\alpha (a)\beta ^2(z))\\
&\overset{(\ref{suplim}), (\ref{BHassoc})}{=}&
(\beta (\alpha ^2\beta (x))\ast \alpha (\alpha ^2\beta (y)))\beta (\beta ^2(z))
-\alpha \beta (\alpha ^2\beta (x))\ast (\alpha (\alpha ^2\beta (y))\beta ^2(z))\\
&&+(\alpha ^2\beta ^2(y)(\alpha ^2\beta (x)\alpha \beta (a)))\beta ^3(z)-
(\alpha ^2\beta ^2(y)\alpha ^3\beta (x))(\alpha \beta ^2(a)\beta ^2(z))\\
&\overset{(\ref{commu}), (\ref{BHassoc})}{=}&
(\beta (\alpha ^2\beta (x))\ast \alpha (\alpha ^2\beta (y)))\beta (\beta ^2(z))
-\alpha \beta (\alpha ^2\beta (x))\ast (\alpha (\alpha ^2\beta (y))\beta ^2(z))\\
&&+(\alpha ^2\beta ^2(y)(\beta ^2(a)\alpha ^3(x)))\beta ^3(z)-
\alpha ^3\beta ^2(y)(\alpha ^3\beta (x)(\alpha \beta (a)\beta (z)))\\
&\overset{(\ref{suplim}), (\ref{BHassoc})}{=}&
(\beta (\alpha ^2\beta (x))\ast \alpha (\alpha ^2\beta (y)))\beta (\beta ^2(z))
-\alpha \beta (\alpha ^2\beta (x))\ast (\alpha (\alpha ^2\beta (y))\beta ^2(z))\\
&&+(\alpha ^2\beta ^2(y)(a\alpha ^3(x)))\beta ^3(z)-
\alpha ^3\beta ^2(y)((\alpha ^2\beta (x)\alpha \beta (a))\beta ^2(z))\\
&\overset{(\ref{suplim}), (\ref{commu})}{=}&
(\beta (\alpha ^2\beta (x))\ast \alpha (\alpha ^2\beta (y)))\beta (\beta ^2(z))
-\alpha \beta (\alpha ^2\beta (x))\ast (\alpha (\alpha ^2\beta (y))\beta ^2(z))\\
&&+(\alpha ^2\beta ^2(y)(a\alpha ^3(x)))\beta ^3(z)-
\alpha ^3\beta ^2(y)((a\alpha ^3(x))\beta ^2(z))\\
&\overset{(\ref{4.1}), (\ref{BHassoc})}{=}&(\alpha ^2\beta ^2(y)\ast \alpha ^3\beta (x))\beta ^3(z)-
\alpha ^3\beta ^2(y)\ast (\alpha ^3\beta (x)\beta ^2(z))\\
&&+(\alpha ^2\beta ^2(y)(a\alpha ^3(x)))\beta ^3(z)-\alpha ^3\beta ^2(y)(\alpha (a)(\alpha ^3(x)\beta (z)))\\
&=&(\alpha ^2\beta ^2(y)\ast \alpha ^3\beta (x)+\alpha ^2\beta ^2(y)(a\alpha ^3(x)))\beta ^3(z)\\
&&-\alpha ^3\beta ^2(y)\ast (\alpha ^3\beta (x)\beta ^2(z))-
\alpha ^3\beta ^2(y)(\alpha (a)(\alpha ^3(x)\beta (z)))\\
&\overset{(\ref{suplim})}{=}&(\alpha \beta ^2(y)\ast \alpha ^2\beta (x)+\alpha \beta ^2(y)(\alpha (a)\alpha ^2(x)))
\cdot _{\alpha , \beta }\beta ^2(z)\\
&&-\alpha (\alpha ^2\beta ^2(y))\ast \beta (\alpha ^3 (x)\beta (z))-
\alpha (\alpha ^2\beta ^2(y))(\alpha (a)(\alpha ^3(x)\beta (z)))\\
&=&(\beta ^2(y)\times \alpha ^2(x))\cdot _{\alpha , \beta }\beta ^2(z)-
\alpha ^2\beta ^2(y)\times (\alpha ^3(x)\beta (z))\\
&=&(\beta ^2(y)\times \alpha ^2(x))\cdot _{\alpha , \beta }\beta ^2(z)-
\alpha ^2\beta ^2(y)\times (\alpha ^2(x)\cdot _{\alpha , \beta }z).
\end{eqnarray*}
Finally, to prove (\ref{1.11}) we compute:
\begin{eqnarray*}
\alpha ^2(x)\cdot _{\alpha , \beta }(y\times z)&=&\alpha ^2(x)\cdot _{\alpha , \beta }(\alpha (y)\ast \beta (z)+
\alpha (y)(\alpha (a)z))\\
&=&\alpha ^3(x)(\alpha \beta (y)\ast \beta ^2(z))+\alpha ^3(x)(\alpha \beta (y)(\alpha \beta (a)\beta (z)))\\
&\overset{(\ref{suplim}), (\ref{BHassoc}), (\ref{1.11})}{=}&(\alpha ^2(x)\alpha \beta (y))\ast \beta ^3(z)+
(\alpha ^2(x)\alpha \beta (y))(\alpha (a)\beta ^2(z))\\
&=&\alpha (\alpha (x)\beta (y))\ast \beta (\beta ^2(z))+\alpha (\alpha (x)\beta (y))(\alpha (a)\beta ^2(z))\\
&=&(\alpha (x)\beta (y))\times \beta ^2(z)=(x\cdot _{\alpha , \beta }y)\times \beta ^2(z),
\end{eqnarray*}
finishing the proof.
\end{proof}

The next result is the Hom-version of Theorem \ref{theorem 2} with $\a=\b$, which is Theorem 4.4 in
\cite{yau3}.
\begin{corollary}\label{Coro 3.6}
Let $(A, \mu , \ast, \alpha )$ be a multiplicative Hom-Novikov-Poisson algebra and $a\in A$ an element satisfying $\a^2(a)=a$.
Then $\overline{A}=(A, \cdot_{\a}, \times, \alpha^2)$ is also a multiplicative Hom-Novikov-Poisson algebra, where
$x\cdot_{\a} y=\a(x)\a(y)$ and
$x\times y=\a(x)\ast \a(y)+ a(xy)$,
for all $x, y\in A$.
\end{corollary}

Forgetting about the BiHom-associative product $\cdot_{\a , \b}$ in Theorem \ref{theorem 2}, we obtain a non-trivial way to construct a BiHom-Novikov algebra from a BiHom-Novikov-Poisson algebra:
\begin{corollary}\label{Coro 3.7}
Let $(A, \mu , \ast, \alpha , \beta )$ be a BiHom-Novikov-Poisson algebra and $a\in A$ an element
satisfying $\a^2(a)=\b^2 (a)=a$.
Then $(A, \times, \alpha^2, \beta^2)$ is a BiHom-Novikov algebra, where
$x\times y = \a(x)\ast \b(y)+ \alpha (x)(\alpha (a)y)$,
for all $x, y\in A$.
\end{corollary}

The following perturbation result is obtained by combining Theorem \ref{theorem 1} and Theorem \ref{theorem 2}.
\begin{corollary}\label{Coro 3.9}
Let $(A, \mu, \ast, \alpha , \beta )$ be a BiHom-Novikov-Poisson algebra and $a, b \in A$ elements such that
$\a^2(a)=\b^2 (a)=a$ and
$\a^4(b)=\b^4 (b)=b$.
Then $\widetilde{A}=(A, \diamond, \boxtimes, \alpha^4, \beta^4)$ is also a BiHom-Novikov-Poisson algebra, where
\begin{eqnarray*}
x\diamond y &=& \alpha ^3(x)(\alpha ^3\beta (b)\beta ^2(y)), \;\;\;\forall \; x, y\in A,\\
x\boxtimes y &=& \a^3(x)\ast \b^3(y)+ \alpha ^3(x)(\alpha (a)\beta ^2(y)), \;\;\;\forall \; x, y\in A.
\end{eqnarray*}
\end{corollary}
\begin{proof}
By Theorem \ref{theorem 2}, we get that $\overline{A}=(A, \cdot_{\a , \b}, \times, \alpha^2, \beta^2)$ is a BiHom-Novikov-Poisson algebra. Now apply Theorem \ref{theorem 1} to $\overline{A}$ and the element $b\in A$, which satisfies $(\a^2)^2(b)=(\b^2)^2(b)=b$. We obtain a BiHom-Novikov-Poisson algebra $(\overline{A})'$, which is $\widetilde{A}$ above.
\end{proof}

The following result is a special case of Corollary \ref{Coro 3.9}.
\begin{corollary}\label{Coro 3.10}
Let $(A, \mu )$ be a commutative and associative algebra, $\a, \b: A \rightarrow A$ two commuting algebra morphisms, and $D: A\rightarrow A$ a derivation such that $D\circ \a=\a \circ D$
and $D\circ \b=\b \circ D$. Let $a, b \in A$ be elements such that $\a^2(a)=\b^2 (a)=a$ and
$\a^4(b)=\b^4 (b)=b$.
Then $(A, \lozenge, \boxdot, \alpha^4, \beta^4)$ is a BiHom-Novikov-Poisson algebra, where
\begin{eqnarray*}
x\lozenge y &=& \alpha ^4(x)\beta ^2(b)\beta ^4(y), \;\;\; \forall \; x, y\in A, \\
x\boxdot y &=& \alpha ^4(x)D(\beta ^4(y))+\alpha ^4(x)\beta (a)\beta ^4(y), \;\;\; \forall \; x, y\in A.
\end{eqnarray*}
\end{corollary}
\begin{proof}
By Corollary 3.10 in \cite{lmmp4}, we get that $A_{\a ,\b}=(A, \bullet, \ast , \alpha, \beta)$ is a BiHom-Novikov-Poisson algebra, where $x\bullet y=\a(x)\b(y)$ and $x \ast y=\a(x)D(\b(y))$. Now apply Corollary \ref{Coro 3.9} to $A_{\a ,\b}$ and the elements $a$ and $b$. The outcome is the BiHom-Novikov-Poisson algebra $\widetilde{A_{\a ,\b}}$, which is exactly
$(A, \lozenge, \boxdot, \alpha^4, \beta^4)$.
\end{proof}
\section{From BiHom-Novikov-Poisson algebras to BiHom-Poisson algebras}\label{sec4}
\setcounter{equation}{0}

The purpose of this section is to investigate a special class of  BiHom-Novikov-Poisson algebras, called left BiHom-associative,
having the property that, in case of bijective structure maps, they provide BiHom-Poisson algebras.

\begin{definition}\label{Def-BiHomLie} (\cite{lmmp2})
A BiHom-Lie algebra $\left( L,
\left[\cdot , \cdot \right] ,\alpha ,\beta \right) $ is a 4-tuple in which $L$ is a linear
space, $\alpha , \beta :L\rightarrow L$ are linear maps and $\left[\cdot , \cdot \right]
:L\times L\rightarrow L$ is a bilinear map,
such that
\begin{gather}
\alpha \circ \beta =\beta \circ \alpha , \label{bihomlie1} \\
\alpha (\left[ x, y \right])=\left[ \alpha \left(
x\right),\alpha (y) \right]\;\;
\text{ and } \;\;\beta (\left[x, y \right])=\left[
\beta \left( x \right) ,\beta \left( y \right) %
\right], \label{bihomlie2} \\
\left[ \beta \left( x\right) ,\alpha \left( y\right) \right] =-%
\left[ \beta \left( y\right) ,\alpha \left( x\right) \right],
\;\;\;\; \text{ (BiHom-skew-symmetry)} \label{bihomlie3}\\
\left[ \beta ^{2}\left( x\right) ,\left[ \beta \left( y\right)
,\alpha \left( z\right) \right] \right] +\left[ \beta
^{2}\left( y\right) ,\left[ \beta \left( z\right)
,\alpha \left( x\right) \right] \right] +\left[ \beta ^{2}\left( z\right) ,\left[ \beta \left( x\right) ,\alpha \left( y
\right) \right] \right] =0, \label{bihomlie4}\\
\text{ (BiHom-Jacobi condition)} \nonumber
\end{gather}
for all $x, y, z\in L$.  
The maps $\alpha $ and $\beta $ (in this order) are called the structure maps
of $L$.

\end{definition}
\begin{definition}\label{Def-BiHomPoisson}
A BiHom-Poisson algebra is a 5-tuple $(A, \mu , \left[\cdot , \cdot \right] , \alpha ,
\beta )$, with the property that

(1) $(A, \mu , \alpha , \beta )$ is a BiHom-commutative algebra;

(2) $(A, \left[\cdot , \cdot \right], \a , \b )$ is a BiHom-Lie algebra;

(3) the following BiHom-Leibniz identity holds for all $x, y, z\in A$:
\begin{eqnarray}
\left[\a\b(x) , yz \right]=\left[\b(x) , y\right]\b(z)+\b(y)\left[\a(x) , z\right].
\label{bihompoisson}
\end{eqnarray}
\end{definition}

BiHom-Poisson algebras are defined as a BiHom-type generalization of Poisson algebras (the case $\a=\b=id_A$)
and Hom-Poisson algebras (the case $\a=\b$, see
\cite{ms2} and \cite{yau3}).

Let us mention that (\ref{bihompoisson}) is actually a left-handed BiHom-Leibniz identity, so the concept introduced
in Definition \ref{Def-BiHomPoisson} is actually a sort of  left-handed BiHom-Poisson algebra, as compared to the
right-handed versions of both that are introduced in \cite{adimi}.

We are interested now in the following question: under what conditions one can construct a BiHom-Poisson algebra from a
BiHom-Novikov-Poisson algebra by taking the commutator bracket of the BiHom-Novikov product (which
automatically requires bijectivity of the structure maps)? We begin by introducing the following definitions.
\begin{definition}\label{Def-left BiHomasso}
Let $(A, \mu , \ast , \alpha ,
\beta )$ be a BiHom-Novikov-Poisson algebra. Then $A$ is called left BiHom-associative if the following condition
holds for all $x, y, z\in A$:
\begin{eqnarray}
&&\a(x)\ast (yz)=(xy)\ast \b(z).
\label{left BiHomasso 1}
\end{eqnarray}
\end{definition}

By using (\ref{1.11}), it is easy to see that (\ref{left BiHomasso 1}) holds if and only if
\begin{eqnarray}
\a(x)(y\ast z)=\a(x)\ast (yz), \;\;\; \forall \; x, y, z\in A.
\label{left BiHomasso 2}
\end{eqnarray}

\begin{definition}\label{Def-admissble}
Let $(A, \mu , \ast , \alpha ,
\beta )$ be a BiHom-Novikov-Poisson algebra with $\a, \b$ bijective. Then $A$ is called admissible if
$A^{-}:=(A, \mu , \left[ \cdot , \cdot \right] , \alpha ,
\beta )$ is a BiHom-Poisson algebra, where
\begin{eqnarray*}
\left[x , y\right]=x\ast y- \a^{-1}\b(y)\ast \a\b^{-1}(x), \;\;\; \forall \; x, y\in A.
\end{eqnarray*}
\end{definition}

We now give a necessary and sufficient condition under which a BiHom-Novikov-Poisson algebra is admissible. The Hom-type
version of this result may be found in \cite{yau3}.
\begin{theorem}\label{theorem 4.5}
Let $(A, \mu , \ast , \alpha ,
\beta )$ be a BiHom-Novikov-Poisson algebra with $\a, \b$ bijective. Then $A$ is admissible if and only if it is left BiHom-associative.
\end{theorem}

\begin{proof}
We know from the definition that
$(A, \mu , \alpha , \beta )$ is a BiHom-commutative algebra, while the fact that
$(A, \left[ \cdot , \cdot \right] , \alpha ,\beta )$ is a BiHom-Lie algebra follows from \cite{lmmp2}, Proposition 2.4.\\
The left-hand side of the BiHom-Leibniz identity (\ref{bihompoisson}) for $A^{-}$ is:
\begin{eqnarray*}
\left[\a\b(x) , yz \right]
&=& \a\b(x)\ast (yz)-\a^{-1}\b(yz)\ast \a^2(x)\\
&=& \a\b(x)\ast (\a(\a^{-1}(y))z)-(\a^{-1}\b(y)\b(\a^{-1}(z)))\ast \a\b(\a\b^{-1}(x))\\
&\overset{(\ref{4.1}), (\ref{4.2})}{=}& (\b(x)\ast y)\b(z)-(\a^{-1}\b(y)\ast \a(x))\b(z)\\
 && +\b(y)\ast (\a(x)z)-\b(y)(\a^{-1}\b(z)\ast \a^2\b^{-1}(x)).
\end{eqnarray*}
The right-hand side of (\ref{bihompoisson}) for $A^{-}$ is:
\begin{eqnarray*}
\left[\b(x) , y\right]\b(z)+\b(y)\left[\a(x) , z\right]
&=& (\b(x)\ast y)\b(z)-(\a^{-1}\b(y)\ast \a(x))\b(z)\\
 && +\b(y)(\a(x)\ast z)-\b(y)(\a^{-1}\b(z)\ast \a^2\b^{-1}(x)).
\end{eqnarray*}
It follows that $A^{-}$ satisfies (\ref{bihompoisson}) if and only if
$\b(y)\ast (\a(x)z)=\b(y)(\a(x)\ast z)$.
Since
\begin{eqnarray*}
\b(y)(\a(x)\ast z)=
\a(\a^{-1}\b(y))(\b(\a\b^{-1}(x))\ast z)\overset{(\ref{1.11})}{=} (\a^{-1}\b(y)\a(x))\ast \b(z)
\end{eqnarray*}
and
$x, y, z \in A$ are arbitrary, we obtain that $A^{-}$ is a BiHom-Poisson algebra if and only if
$\a(y)\ast (xz)=(yx)\ast \b(z)$ holds, that is, if and only if $A$ is
left BiHom-associative.
\end{proof}

\begin{example}
Let $(A, \mu)$ be a commutative and associative algebra, $\alpha, \beta: A\rightarrow A$ two commuting algebra morphisms,
and $D: A\rightarrow A$ a derivation such that $D\circ \alpha=\alpha\circ D$ and
$D\circ \beta=\beta\circ D$. By Corollary 3.10 in \cite{lmmp4},
$A_{\a ,\b}=(A, \bullet, \ast , \alpha, \beta)$ is a BiHom-Novikov-Poisson algebra, where $x\bullet y=\a(x)\b(y)$ and $x \ast y=\a(x)D(\b(y))$, for all $x, y\in A$. Then $A_{\a, \b}$ is left BiHom-associative if and only if
$ \a ^2(x)D(\a\b(y))\beta ^2 (z)=0$, for all $x, y, z\in A$.

Indeed, $A_{\a, \b}$ is left BiHom-associative if and only if
(\ref{left BiHomasso 2}) holds for $A_{\a, \b}$, i.e. if and only if
\begin{eqnarray*}
0&=&\a(x)\ast(y\bullet z)-\a(x)\bullet(y\ast z)\\
&=& \alpha (x)\ast (\alpha (y)\beta (z))-\alpha (x)\bullet (\alpha (y)D(\beta (z)))\\
&=&\alpha ^2(x)D(\alpha \beta (y)\beta ^2(z))-\alpha ^2(x)\alpha \beta (y)D(\beta ^2(z))\\
&=&\alpha ^2(x)D(\alpha \beta (y))\beta ^2(z)+\alpha ^2(x)\alpha \beta (y)D(\beta ^2(z))
-\alpha ^2(x)\alpha \beta (y)D(\beta ^2(z))\\
&=&\alpha ^2(x)D(\alpha \beta (y))\beta ^2(z).
\end{eqnarray*}
\end{example}

Now we prove that left BiHom-associativity is preserved by Yau twisting.
\begin{proposition} \label{coro 4.7}
Let $(A, \mu , \ast, \alpha, \beta)$ be a left BiHom-associative BiHom-Novikov-Poisson algebra and $\tilde{\alpha}, \tilde{\beta}: A\rightarrow A$ two  morphisms of BiHom-Novikov-Poisson algebras
such that any two of the maps $\alpha, \beta, \tilde{\alpha}, \tilde{\beta}$ commute. Then
$A_{(\tilde{\alpha}, \tilde{\beta})}:=(A, \tilde{\cdot}:=\mu \circ (\tilde{\alpha}\otimes \tilde{\beta}),\, \tilde{\ast}:=\ast \circ (\tilde{\alpha}\otimes \tilde{\beta}),\, \alpha\circ \tilde{\alpha}, \, \beta\circ \tilde{\beta})$
 is also a left BiHom-associative BiHom-Novikov-Poisson algebra.
\end{proposition}
\begin{proof}
From \cite{lmmp4}, we know that $A_{(\tilde{\alpha}, \tilde{\beta})}$ is a BiHom-Novikov-Poisson algebra.
For all $x, y, z \in A$,\\[2mm]
${\;\;\;\;\;\;\;\;\;\;}$
$(x \widetilde{\cdot} y)\widetilde{\ast} \b\widetilde{\b}(z)-\a\widetilde{\a}(x)\widetilde{\ast} (y \widetilde{\cdot} z)$
\begin{eqnarray*}
&=& (\widetilde{\a}(x) \widetilde{\b}(y))\widetilde{\ast} \b\widetilde{\b}(z)-\a\widetilde{\a}^2(x)\ast \widetilde{\b}(y \widetilde{\cdot} z)\\
&=& (\widetilde{\a}^2(x) \widetilde{\a}\widetilde{\b}(y))\ast \b\widetilde{\b}^2(z)-\a\widetilde{\a}^2(x)\ast (\widetilde{\a}\widetilde{\b}(y)  \widetilde{\b}^2(z))\\
&=& (\widetilde{\a}^2(x) \widetilde{\a}\widetilde{\b}(y))\ast \b(\widetilde{\b}^2(z))-\a(\widetilde{\a}^2(x))\ast (\widetilde{\a}\widetilde{\b}(y)  \widetilde{\b}^2(z))
\overset{(\ref{left BiHomasso 1})}{=} 0,
\end{eqnarray*}
thus proving  (\ref{left BiHomasso 1}) for
$A_{(\tilde{\alpha}, \tilde{\beta})}$ and finishing the proof.
\end{proof}

In the context of Proposition \ref{coro 4.7} and assuming moreover that $\alpha, \beta, \tilde{\alpha}, \tilde{\beta}$
are bijective,
the BiHom-Lie bracket in the BiHom-Poisson algebra $(A_{(\tilde{\alpha}, \tilde{\beta})})^{-}$ is given by
$\widetilde{\left[\cdot , \cdot \right]}=\left[\cdot , \cdot  \right]\circ (\widetilde{\a} \otimes \widetilde{\b})$,
where $\left[\cdot , \cdot \right]$ is the BiHom-Lie bracket in the BiHom-Poisson algebra $A^{-}$.

The next result is a special case of  Proposition \ref{coro 4.7}.
\begin{corollary} \label{coro 4.8}
Let $(A, \mu, \ast, \alpha, \beta)$ be a left BiHom-associative BiHom-Novikov-Poisson algebra. Then so is
$A^{n}:=(A, \mu \circ (\a^n\otimes \b^n),\, \ast \circ (\a^n\otimes \b^n),\, \alpha^{n+1}, \, \beta^{n+1})$,
 for each $n\geq 0$.
\end{corollary}
\begin{proof}
Apply Proposition \ref{coro 4.7} for $\tilde{\alpha }:=\alpha^{n}$ and
$\tilde{\beta }:=\beta^{n}$.
\end{proof}

In the rest of this section, we show that left BiHom-associativity is compatible with the constructions in the previous sections. 
First we deal with tensor products.
\begin{proposition} \label{coro 4.9}
Let $(A_i, \cdot_i, \ast_i, \alpha_i, \beta_i)$ be a left BiHom-associative BiHom-Novikov-Poisson algebra, for $i=1, 2$, and let $A=A_1\otimes A_2$ be the BiHom-Novikov-Poisson algebra constructed in Theorem
\ref{tensor}. Then $A$ is left BiHom-associative.
\end{proposition}
\begin{proof}
Pick $x=x_1\otimes x_2, y=y_1\otimes y_2$ and $z=z_1\otimes z_2$ in $A$. Then we have:\\[2mm]
$[(x_1\ot x_2)\cdot (y_1\ot y_2)]\ast (\b_1(z_1)\ot \b_2(z_2))$
\begin{eqnarray*}
&=& [(x_1\cdot y_1)\ot (x_2\cdot y_2)]\ast (\b_1(z_1)\ot \b_2(z_2))\\
&=& ((x_1\cdot y_1)\ast \b_1(z_1))\ot ((x_2\cdot y_2)\cdot \b_2(z_2))+((x_1\cdot y_1)\cdot \b_1(z_1))\ot ((x_2\cdot y_2)\ast \b_2(z_2))\\
&\overset{(\ref{BHassoc}), (\ref{left BiHomasso 1})}{=}& (\a_1(x_1)\ast (y_1\cdot z_1))\ot (\a_2(x_2)\cdot (y_2\cdot z_2))+(\a_1(x_1)\cdot (y_1\cdot z_1))\ot (\a_2(x_2)\ast (y_2\cdot z_2))\\
&=& (\a_1(x_1)\ot \a_2(x_2))\ast [(y_1\ot y_2)\cdot (z_1\ot z_2)],
\end{eqnarray*}
thus proving (\ref{left BiHomasso 1}) for $A$ and
finishing the proof.
\end{proof}

In the context of Proposition \ref{coro 4.9} and assuming moreover that all structure maps are bijective,
the BiHom-Lie bracket in the BiHom-Poisson algebra $A^{-}$ is given by \\[2mm]
$\left[x_1\ot x_2, y_1\ot y_2 \right]$
\begin{eqnarray*}
&=&(x_1\ot x_2)\ast (y_1\ot y_2)-(\a_1^{-1}\b_1(y_1)\ot \a_2^{-1}\b_2(y_2))\ast (\a_1\b_1^{-1}(x_1)\ot \a_2\b_2^{-1}(x_2))\\
&=& (x_1\ast y_1)\otimes (x_2\cdot y_2) + (x_1\cdot y_1)\otimes (x_2\ast y_2)\\
&& -(\b_1(\a_1^{-1}(y_1))\cdot \a_1(\b_1^{-1}(x_1)))\otimes (\a_2^{-1}\b_2(y_2)\ast \a_2\b_2^{-1}(x_2))\\
&& -(\a_1^{-1}\b_1(y_1)\ast \a_1\b_1^{-1}(x_1))\otimes (\b_2(\a_2^{-1}(y_2))\cdot \a_2(\b_2^{-1}(x_2)))\\
&\overset{(\ref{commu})}{=}& (x_1\ast y_1)\otimes (x_2\cdot y_2) + (x_1\cdot y_1)\otimes (x_2\ast y_2)\\
&& - (x_1\cdot y_1)\otimes (\a_2^{-1}\b_2(y_2)\ast \a_2\b_2^{-1}(x_2))-(\a_1^{-1}\b_1(y_1)\ast \a_1\b_1^{-1}(x_1))\otimes (x_2\cdot y_2)\\
&=& (x_1\ast y_1- \a_1^{-1}\b_1(y_1)\ast \a_1\b_1^{-1}(x_1))\otimes (x_2\cdot y_2)\\
&& + (x_1\cdot y_1)\otimes (x_2\ast y_2 - \a_2^{-1}\b_2(y_2)\ast \a_2\b_2^{-1}(x_2))\\
&=& [x_1 , y_1]\otimes (x_2\cdot y_2)+ (x_1\cdot y_1)\otimes [x_2, y_2].
\end{eqnarray*}
Here $\left[x_i , y_i\right]$ is the BiHom-Lie bracket in the BiHom-Poisson algebra $A_i ^{-}$.

Next we prove that left BiHom-associativity is preserved by perturbations as in Theorem \ref{theorem 1}.

\begin{proposition} \label{coro 4.10}
Let $(A, \mu , \ast, \alpha , \beta )$ be a left BiHom-associative BiHom-Novikov-Poisson algebra  and $a\in A$ an element
satisfying $\a^2(a)=\b^2 (a)=a$.
Then the BiHom-Novikov-Poisson algebra $A'=(A, \diamond, \ast_{\a, \b}, \alpha^2, \beta^2)$ constructed in Theorem \ref{theorem 1} is also left BiHom-associative.
\end{proposition}
\begin{proof}
We need to prove (\ref{left BiHomasso 1}) for $A'$. We compute: \\[2mm]
${\;\;\;\;\;}$
$\alpha ^2(x)\ast _{\alpha , \beta }(y\diamond z)$
\begin{eqnarray*}
&=&\alpha ^2(x)\ast _{\alpha , \beta }(\alpha (y)(\alpha (a)z))=
\alpha ^3(x)\ast (\alpha \beta (y)(\alpha \beta (a)\beta (z)))\\
&\overset{(\ref{left BiHomasso 1})}{=}&(\alpha ^2(x)\alpha \beta (y))\ast \beta (\alpha \beta (a)\beta (z))
\overset{(\ref{suplim})}{=}(\alpha ^2(x)\alpha \beta (y))\ast (\alpha (a)\beta ^2(z))\\
&=&\alpha (\alpha (x)\beta (y))\ast (\alpha (a)\beta ^2(z))\overset{(\ref{left BiHomasso 1})}{=}
((\alpha (x)\beta (y))\alpha (a))\ast \beta ^3(z)\\
&\overset{(\ref{suplim})}{=}& ((\alpha (x)\beta (y))\alpha \beta ^2(a))\ast \beta ^3(z)
\overset{(\ref{BHassoc})}{=}(\alpha ^2(x)(\beta (y)\alpha \beta (a)))\ast \beta ^3(z)\\
&\overset{(\ref{commu})}{=}&(\alpha ^2(x)(\beta ^2(a)\alpha (y)))\ast \beta ^3(z)
\overset{(\ref{suplim})}{=} (\alpha ^2(x)(\alpha ^2(a)\alpha (y)))\ast \beta ^3(z)\\
&=&\alpha (\alpha (x)(\alpha (a)y))\ast \beta ^3(z)=\alpha (x\diamond y)\ast \beta ^3(z)
=(x\diamond y)\ast _{\alpha , \beta }\beta ^2(z),
\end{eqnarray*}
finishing the proof.
\end{proof}

In the context of Proposition \ref{coro 4.10} and assuming moreover that $\alpha $ and $\beta $ are bijective,
 the BiHom-Lie bracket in the BiHom-Poisson algebra $(A')^{-}$ is given by
\begin{eqnarray*}
x\ast_{\a, \b} y- \a^{-2}\b^2(y)\ast_{\a, \b}\a^2\b^{-2}(x)
 &=& \a(x)\ast \b(y)- \a^{-1}\b^2(y)\ast \a^2\b^{-1}(x)\\
 &=& \a(x)\ast \b(y)- \a^{-1}\b(\b(y))\ast \a\b^{-1}(\a(x))\\
 &=& \left[\a(x), \b(y) \right]=\left[\cdot  , \cdot \right]\circ (\a\otimes \b)(x\otimes y),
\end{eqnarray*}
where $\left[\cdot , \cdot \right]$ is the BiHom-Lie bracket in the BiHom-Poisson algebra $A^{-}$.

Finally, we prove a similar result for perturbations as in Theorem \ref{theorem 2}.
\begin{proposition} \label{lastprop}
Let $(A, \mu , \ast , \alpha , \beta )$ be a left BiHom-associative BiHom-Novikov-Poisson algebra and $a\in A$
an element satisfying $\alpha ^2(a)=\beta ^2(a)=a$. Then the BiHom-Novikov-Poisson algebra
$\overline{A}=(A, \cdot _{\alpha , \beta }, \times , \alpha ^2, \beta ^2)$ constructed in Theorem \ref{theorem 2}
is also left BiHom-associative.
\end{proposition}
\begin{proof}
We need to prove (\ref{left BiHomasso 1}) for $\overline{A}$. We compute:
\begin{eqnarray*}
\alpha ^2(x)\times (y\cdot _{\alpha , \beta }z)&=&\alpha ^2(x)\times (\alpha (y)\beta (z))\\
&=&\alpha ^3(x)\ast (\alpha \beta (y)\beta ^2(z))+\alpha ^3(x)(\alpha (a)(\alpha (y)\beta (z)))\\
&\overset{(\ref{left BiHomasso 1}), (\ref{BHassoc})}{=}&
(\alpha ^2(x)\alpha \beta (y))\ast \beta ^3(z)+\alpha ^3(x)((a\alpha (y))\beta ^2(z))\\
&\overset{(\ref{suplim})}{=}&(\alpha ^2(x)\alpha \beta (y))\ast \beta ^3(z)
+\alpha ^3(x)((\beta ^2(a)\alpha (y))\beta ^2(z))\\
&\overset{(\ref{commu})}{=}&(\alpha ^2(x)\alpha \beta (y))\ast \beta ^3(z)
+\alpha ^3(x)((\beta (y)\alpha \beta (a))\beta ^2(z))\\
&\overset{(\ref{BHassoc})}{=}&(\alpha ^2(x)\alpha \beta (y))\ast \beta ^3(z)
+\alpha ^3(x)(\alpha \beta (y)(\alpha \beta (a)\beta (z)))\\
&\overset{(\ref{suplim}), (\ref{BHassoc})}{=}&(\alpha ^2(x)\alpha \beta (y))\ast \beta ^3(z)
+(\alpha ^2(x)\alpha \beta (y))(\alpha (a)\beta ^2(z))\\
&=&\alpha (\alpha (x)\beta (y))\ast \beta ^3(z)+\alpha (\alpha (x)\beta (y))(\alpha (a)\beta ^2(z))\\
&=&(\alpha (x)\beta (y))\times \beta ^2(z)=(x\cdot _{\alpha , \beta }y)\times \beta ^2(z),
\end{eqnarray*}
finishing the proof.
\end{proof}

In the context of Proposition \ref{lastprop} and assuming moreover that $\alpha $ and $\beta $ are bijective, the
BiHom-Lie bracket in the BiHom-Poisson algebra $(\overline{A})^{-}$ is given by \\[2mm]
$x\times y-\alpha ^{-2}\beta ^2(y)\times \alpha ^2\beta ^{-2}(x)$
\begin{eqnarray*}
&=&\alpha (x)\ast \beta (y)+\alpha (x)(\alpha (a)y)-\alpha ^{-1}\beta ^2(y)\ast \alpha ^2\beta ^{-1}(x)
-\alpha ^{-1}\beta ^2(y)(\alpha (a)\alpha ^2\beta ^{-2}(x))\\
&=&\alpha (x)\ast \beta (y)-\alpha ^{-1}\beta (\beta (y))\ast \alpha \beta ^{-1}(\alpha (x))
+\alpha (x)(\alpha (a)y)-\alpha (\alpha ^{-2}\beta ^2(y))(\alpha (a)\alpha ^2\beta ^{-2}(x))\\
&\overset{(\ref{BHassoc})}{=}&\alpha (x)\ast \beta (y)-\alpha ^{-1}\beta (\beta (y))\ast \alpha \beta ^{-1}(\alpha (x))
+\alpha (x)(\alpha (a)y)-(\alpha ^{-2}\beta ^2(y)\alpha (a))\alpha ^2\beta ^{-1}(x)\\
&\overset{(\ref{commu})}{=}&\alpha (x)\ast \beta (y)-\alpha ^{-1}\beta (\beta (y))\ast \alpha \beta ^{-1}(\alpha (x))
+\alpha (x)(\alpha (a)y)-(\beta (a)\alpha ^{-1}\beta (y))\alpha ^2\beta ^{-1}(x)\\
&=&\alpha (x)\ast \beta (y)-\alpha ^{-1}\beta (\beta (y))\ast \alpha \beta ^{-1}(\alpha (x))
+\alpha (x)(\alpha (a)y)-\beta (a\alpha ^{-1}(y))\alpha (\alpha \beta ^{-1}(x))\\
&\overset{(\ref{commu})}{=}&\alpha (x)\ast \beta (y)-\alpha ^{-1}\beta (\beta (y))\ast \alpha \beta ^{-1}(\alpha (x))
+\alpha (x)(\alpha (a)y)-\alpha (x)(\alpha (a)y)\\
&=&\alpha (x)\ast \beta (y)-\alpha ^{-1}\beta (\beta (y))\ast \alpha \beta ^{-1}(\alpha (x))
=\left[\a(x), \b(y) \right]=\left[\cdot  , \cdot \right]\circ (\a\otimes \b)(x\otimes y),
\end{eqnarray*}
where $\left[\cdot , \cdot \right]$ is the BiHom-Lie bracket in the BiHom-Poisson algebra $A^{-}$.


\begin{center}
ACKNOWLEDGEMENTS
\end{center}

This paper was written while Claudia Menini was a member of the "National Group for Algebraic and
Geometric Structures and their Applications" (GNSAGA-INdAM). She was partially supported by MIUR
within the National Research Project PRIN 2017.
Ling Liu was supported by the NSF of China (Nos. 11801515, 11601486),
the Natural Science Foundation of Zhejiang Province (No. LY20A010003) and
the Foundation of Zhejiang Educational Committee (No. Y201942625).

\end{document}